\documentclass[10pt,a4paper]{article}
\usepackage{amsmath}                    
\usepackage{amssymb}                    
\usepackage{amsthm}                     
\usepackage{amsfonts}                   
\usepackage{subcaption}
\usepackage[usenames, dvipsnames]{color}
\usepackage{graphics,graphicx}
\usepackage{enumerate}
\usepackage{multirow}
\usepackage{algorithm}
\usepackage{algpseudocode}
\usepackage{dsfont}
\usepackage{float}
\usepackage{mathtools}
\usepackage{bbm,amsbsy}
\usepackage[hidelinks]{hyperref}
\usepackage{lipsum}

\allowdisplaybreaks

\newtheorem{theorem}{Theorem}[section]

\newtheorem{proposition}[theorem]{Proposition}
\newtheorem{lemma}[theorem]{Lemma}
\newtheorem{remark}[theorem]{Remark}

\renewcommand{\epsilon}{\varepsilon}

\newcommand{\R}{\mathbb{R}}
\newcommand{\bigo}{\mathcal{O}}

\def \IC{\mathbb C}

\def \calA{\mathcal{A}}

\def \calJ{\mathcal{J}}

\newcommand{\calS}{\mathcal{S}}

\renewcommand{\R}{\mathbb{R}}

\newcommand{\norm}[1]{\left\lVert #1 \right\rVert}

\usepackage[symbol]{footmisc}

\title{
	Explicit stabilized implementation of singly diagonally implicit Runge-Kutta methods
}

\author{ 
	Ibrahim Almuslimani\footnotemark[1], Gilles Vilmart\footnotemark[2], and Konstantinos Zygalakis\footnotemark[3]
}

\begin{document}

\maketitle
\footnotetext[1]{\'Ecole Polytechnique F\'ed\'erale de Lausanne (EPFL), Swiss Plasma Center (SPC), CH-1015 Lausanne, Switzerland. Ibrahim.Almuslimani@epfl.ch}
\footnotetext[2]{Section de Math\'ematiques, Universit\'e de Gen\`eve, CP 64, 1211 Gen\`eve 4, Switzerland. Gilles.Vilmart@unige.ch}
\footnotetext[3]{School of Mathematics and the Maxwell Institute for Mathematical Sciences, University of Edinburgh, Edinburgh, EH9 3FD, UK. K.Zygalakis@ed.ac.uk}
\begin{abstract}
Implicit methods are a natural approach for the integration of stiff differential equations, to avoid time-step restrictions faced by standard explicit integrators. Explicit stabilised integrators are an alternative to implicit methods, which can be particularly efficient in high-dimensional applications with diffusive terms. Towards the best of both worlds, we introduce a new explicit stabilised implementation of a class of diagonally implicit Runge-Kutta methods. This allows us to implement high-order singly-diagonally-implicit Runge-Kutta methods for advection-diffusion-reaction PDEs with a provable computational cost analogous to that of standard explicit stabilised methods. The main ingredient is to recast the implicit Runge–Kutta update as the steady state of a modified auxiliary system, which is then computed using a partitioned Runge–Kutta–Chebyshev method inspired by optimisation techniques.

\smallskip

\end{abstract}

 \section{Introduction}

In this paper, we consider systems of ordinary differential equations (ODEs) representing space discretisations of partial differential equations (PDEs) of the form
\begin{subequations}\label{eq:main}
\begin{eqnarray}
    \dot{y} &=& F(y), \\
    F(y) &:=& F_{D}(y)+F_{A}(y)+F_{R}(y).
\end{eqnarray}
\end{subequations}
where $F_{D}(y),F_{A}(y),F_{R}(y) \in \mathbb{R}^{d}$ represent diffusion, advection and reaction terms respectively. 
Typically \eqref{eq:main} arises from the spatial discretisation of an advection-diffusion-reaction problem in $N$ spatial dimensions, and assuming a spatial grid of size $\Delta x$, the dimension of the ODE grows like $d =\mathcal{O}((\Delta x)^{-N})$. 

The numerical time integration of \eqref{eq:main} is very challenging since this is a \emph{stiff} ODE. {Following \cite{HaW96}}, we call a differential equation \emph{stiff} when
{a standard explicit numerical integrator such as a standard explicit Runge-Kutta method faces a severe time-step restriction}.  In particular, in the case of \eqref{eq:main} the eigenvalues of the Jacobian of $F_{D}$ are typically distributed along the negative real axis in an interval that grows as $[-\mathcal{O}(\Delta x^{-2}),0]$ (for a symmetric diffusion operator), which implies that the time step $\Delta t$ that can be used by the explicit Euler method {suffers from the CFL} $\Delta t \leq C \Delta x^{2}$. Furthermore, the eigenvalues of the Jacobian of $F_{A}$ are typically distributed along the {imaginary} axis in an interval that grows as  $[- i\mathcal{O}(\Delta x^{-1}), i \mathcal{O}(\Delta x^{-1})]$, while the eigenvalues of the Jacobian of the reaction term $F_{R}$ even though are usually not related to $\Delta x$ can have a ratio $\max_{j}|\mathcal{R} \lambda_{j}|/\min_{j}|\mathcal{R} \lambda_{j}|$ that can be very large or vary over several orders of magnitude {when modeling multiscale reaction terms}.



There are two different approaches to dealing with the issue of stiffness within the class of Runge-Kutta methods. The first relates to the design of implicit methods, with a prime example being the implicit Euler method. These methods have the 
{advantage of achieving stability unconditionally with respect to the stepsize,} 
which comes with the cost of solving a system of nonlinear equations.  When the dimension of the corresponding differential equation is low, the nonlinear system can be solved {efficiently using Newton's method \cite{NW06}}. As the dimension of the differential equation increases, diagonally implicit Runge-Kutta methods are attractive as they lead to a triangular system which is easier to solve using lower-dimensional Jacobian matrices using an appropriate partitioning of the coefficient matrix as proposed in \cite{VSC92}. Nevertheless, when one deals with differential equations like \eqref{eq:main}, dealing with the corresponding nonlinear equations (even in the triangular form) can be cumbersome as it requires specialised linear algebra techniques, {where the use of preconditioners raises issues with both memory access and storage on parallel supercomputers.}

An alternative approach to using implicit methods for solving \emph{stiff} differential equations is to use explicit stabilised methods \cite{Abdrev}. In their basic form, these are first-order explicit methods that can take larger time steps than the standard Euler method (for an easily computable numerical cost). This property together with the simplicity of implementation has made explicit stabilised methods the method of choice for diffusion-dominated problems. In particular, the second-order Runge Kutta Chebyshev (RKC) method was very successful for large dimensional parabolic problems, while methods with nearly optimal extended stability domain sizes were proposed as the second-order ROCK2 method \cite{AbM01} and the fourth-order ROCK4 method \cite{Abd02}. In the spirit of implicit-explicit methods {(IMEX) \cite{ARW95,ARW97}}, the versatility of explicit stabilised methods allowed many extensions to partitioned Runge-Kutta methods for addressing problems with additional advection and possibly stiff reaction terms. Examples include {Implicit Runge Kutta Chebyshev (IRKC)  \cite{SSV06}} to deal with stiff reaction terms and Partitioned Runge Kutta Chebyshev (PRKC) \cite{Zb11} to include advection terms. In addition, the Partitioned Implicit Orthogonal Runge Kutta Chebyshev method  (PIROCK) \cite{AV13} combined these two approaches to propose a second-order method for solving advection-diffusion reaction equations.  
In particular, PIROCK has recently emerged as an efficient alternative to classical splitting approaches for solar atmosphere simulations in astrophysics \cite{WVMP24} and for large-scale gyrokinetic simulations in plasma physics \cite{AUB26,AOUB26}.
Another recent use of explicit stabilised methods is for solving optimisation problems. More precisely, in  \cite{EVVZ21} by exploiting the good stability properties of explicit stabilised methods a new method called the \emph{Runge–Kutta Chebyshev} descent (RKCD) was proposed and showed to perform similarly to the state of the art for strongly convex optimisation problems.

A major limitation of methods such as IRKC, PRKC, and PIROCK that use an explicit stabilised method to deal with the $F_{D}$ term in \eqref{eq:main}, is that they {are} at most second-order accurate in $\Delta t$. {Indeed, while the construction of the PIROCK method which is based on a splitting into three terms needed to solve 12 order conditions, similar construction of higher order methods would be cumbersome due to the dramatic increase of order conditions for partitioned Runge-Kutta methods. }
On the other hand, implicit methods can achieve 
much higher order of accuracy in $\Delta t$, 
but as discussed above the implementation as 
well as the storage requirements of {the Newton 
method and variants} in high dimensions become an issue.  
In this paper, inspired by the success of RKCD 
as an optimisation method, we revisit diagonally implicit 
methods and their implementation. In 
particular, instead of solving the nonlinear 
system with variants of Newton's method, we 
introduce a new variant of the RKCD method 
that only {outputs} first-order information.  This method is simple to implement due to its explicit nature which circumvents the 
need for specialised linear algebra tools.
Furthermore,  it leads to rigorous estimates of the computational complexity of the proposed method. In particular, we can show that
\begin{enumerate}
\item In the pure linear diffusion case ($F_{A}=F_{R}=0$ in \eqref{eq:main}), { we prove that one time step} of an {SDIRK} method\footnote{ Here $\Delta t$ is the time step used by the {SDIRK} method.} (of arbitrary order) by an iterative method requires  $\mathcal{O}(\sqrt{\Delta t}/\Delta x)$ evaluations of $F_{D}$ avoiding the standard CFL condition of standard explicit Runge Kutta methods. This essentially implies that we are able to implement a higher order diagonally implicit Runge-Kutta method at the cost of an explicit stabilised method. 
\item {For a class of linear advection-diffusion PDE ($F_{R}=0$), we prove that in spite of relatively large Peclet numbers the proposed partitioning allows to implement one step of an SDIRK method with a similar complexity to the one of the pure diffusion case discussed at the previous point}.
\item {For various nonlinear diffusion-advection-reaction problems, numerical experiments illustrate the robustness of the approach where we show that the explicit stabilised iterations converge in the case of an order 4 SDIRK method with error control, a state-of-the-art Runge-Kutta method proposed in \cite[Table IV.6.5, p. 100]{HaW96} and considered in this paper to illustrate the new approach.}
\end{enumerate}

The paper is organized as follows. In Section \ref{sec:pre} we revisit the ideas of numerical stability of Runge-Kutta methods, while also discussing the basics of explicit stabilised methods. In Section \ref{sec:main} we revisit implicit Runge-Kutta methods and discuss the main idea behind solving the nonlinear system using an explicit stabilised algorithm. In addition, we discuss a specific application of the new explicit implementation of implicit methods in the context of advection-diffusion-reaction equations and present an algorithm for doing so. {A detailed convergence analysis {in the case of pure linear diffusion and linear diffusion-advection} is presented in Section \ref{sec:conv}}. 
Finally, in Section \ref{sec:num} we present a number of numerical computations that illustrate the theoretical convergence analysis in the pure linear diffusion case, and the linear diffusion-advection case. {In addition, applying the new explicit stabilised implementation exSDIRK4 of an order 4 SDIRK method to several nonlinear advection-diffusion-reaction problems illustrates the wide applicability and versatility of our approach. }



        \section{Preliminaries}
\label{sec:pre}
In this paper, we are interested in the numerical integration of systems of ODEs of the form 
\begin{equation}\label{eq:ode}
		\dot y(t)=g(y(t)),\qquad y(0)=y_0, \quad 
t>0	\end{equation}
where $y(t) \in \mathbb{R}^{d}$ and the vector field $g: \mathbb{R}^{d} \mapsto \mathbb{R}^{d}$ is assumed smooth.  In this section, we recall the main tools for stability and accuracy analysis of Runge-Kutta methods.


        \subsection{Stability function and stability domain}
        
        Let $\lambda\in \IC $ and consider the Dahlquist test equation
        \begin{equation}\label{eq:testode}
            \dot y(t) = \lambda y(t), \quad y(0)=y_0\in \R.
        \end{equation}
        In the case where $\lambda \in \IC^-=\{z \in \IC;\,\text{Re}(z) \leq 0\}$, equation \eqref{eq:testode} has a bounded solution for  $t \rightarrow +\infty$. The simplest Runge-Kutta method for solving \eqref{eq:ode} is the explicit Euler method which when applied to \eqref{eq:testode} yields the recurrence
        \begin{equation}
            y_{n+1}=(1+\lambda \Delta t)y_{n}
        \end{equation}
        More generally, a Runge-Kutta method with step size $\Delta t$ applied to \eqref{eq:testode} produces the following recurrence
        \begin{equation} \label{eq:rec_dal}
        y_{n}=R(\Delta t\lambda)y_{n-1}=R(\Delta t \lambda)^ny_0,
        \end{equation}
        where $R(z)$ is called the stability function and the stability domain is defined by
        \begin{equation*}
            \calS=\{z\in\IC\,;\,|R(z)|\leq 1\}.
        \end{equation*}
        Note that $z \in \calS$ ensures that $y_{n}$ in \eqref{eq:rec_dal} remains bounded as $n \rightarrow \infty$ similarly to the true solution. In the case of the explicit Euler we have that 
        $
        \mathcal{S}_{\text{E}}=\{|z+1|\leq 1 \}
        $
        which corresponds to a disc of radius one, centred at $-1$ on the complex plane. For stiff problems where $|\lambda| \gg 1$ this implies a severe time-step restriction namely $\Delta t < 2/|\lambda|$ when $\lambda <0$. In particular, for spatial discretization of the heat equation eigenmodes are typically of size $\mathcal{O}(\Delta x ^{-2})$ which would in turn imply the CFL condition $\Delta t  < c \Delta x^{2}$.

        To avoid such severe time-step restrictions one solution is to use implicit methods. The simplest one of those is the implicit Euler method which when applied to  \eqref{eq:testode} yields the recurrence
        \[
        y_{n+1}=\frac{1}{1-\lambda \Delta t} y_{n}.
        \]
        In this  case we have $\mathcal{S}_{\text{IE}}=\{|z-1| \geq 1 \}$ which contains $\mathbb{C}^{-}$. This property is known as A-stability and can only hold in the case of implicit methods.  A-stability means that there is no restriction on the time-step size for eigenvalues with negative real parts contrary to the {explicit} Euler method and other explicit methods. Nevertheless, in the case of severely stiff  dissipative problems, a desirable additional property to A-stability is $R(\infty)=0$ which is known as L-stability \cite{HaW96}.

        

       \subsection{Higher order Implicit Runge-Kutta methods}
       \label{subsec:hoirk}
        The implicit Euler is an example of an $L$-stable integrator meaning that is particularly suitable for integrating stiff equations. However, from the point of view of accuracy, it is a first-order method. To address this issue families of $L$-stable {Runge-Kutta methods have been proposed \cite{HaW96}} with an arbitrarily higher order of accuracy. Important examples are the families of the Radau and Lobatto Runge-Kutta methods are that achieve orders $2m-1$ and $2m-2$ respectively for any number $m$ of {internal stages \cite{I96}}. In particular, Radau with $m=1$ reduces to the implicit Euler method, while in the case where $m=2$  we obtain {the following implicit Runge-Kutta method for integrating \eqref{eq:ode},}
        \begin{subequations} \label{eq:radau2}
        \begin{eqnarray} 
         Y_{1}&=& y_{n}+\frac{5\Delta t}{12 }g\left(Y_{1}\right)-\frac{\Delta t}{12}g(Y_{2}), \\
       Y_{2}&=&  y_{n}+\frac{3\Delta t}{4 }g\left(Y_{1}\right)+\frac{\Delta t}{4}g(Y_{2}), \\
       y_{n+1}&=& Y_{2}.
       \end{eqnarray}
       \end{subequations}
  For a general $m$-stage implicit Runge-Kutta method with coefficients $a_{ij},b_j, i,j=1,\ldots,m$, given the numerical solution $y_n$ at time $t_{n}$  one needs to solve the following nonlinear system in dimension $m\cdot d$ with unknowns the internal stages $Y_i\in\R^d,i=1,\ldots,m$,
        \begin{equation}\label{eq:nonlinRK}
            Y_i= y_n + \Delta t \sum_{j=1}^m a_{ij} g(Y_j),\quad i=1,\ldots,m,
        \end{equation}
        to obtain the solution  $y_{n+1}$ at time $t_{n+1}$ by 
        \begin{equation} \label{eq:nonlinRKfinal}
            y_{n+1}= y_n + \Delta t \sum_{i=1}^m b_i g(Y_i).
        \end{equation}
        It is common to represent Runge-Kutta methods using a Butcher tableau
        \begin{equation}
            \begin{array}{c|c}
                 c&\calA\\
                 \hline
                 & b
            \end{array}
        \end{equation}
        where $\calA$ is the matrix of coefficients $\{a_{ij}\}_{i,j=1}^{m}$, $b$ is the weights vector $\{b_i\}_{i=1}^{m}$, and the vector $c$ is such that $c_i=\sum_{j=1}^{m} a_{ij}$.
        In compact matrix and tensor notations, \eqref{eq:nonlinRK} can be written as
\begin{equation}\label{eq:nonlinearimp}
	\boldsymbol{Y}
	=
	\mathbbm{1}_m \otimes y_n
	+
	\Delta t (\calA \otimes I_d)
	g[\boldsymbol{Y}]
	\end{equation}
where $\calA \in \R^{m\times m}$ is the Runge-Kutta matrix of coefficients $a_{ij}$ and 
\begin{equation} \label{eq:summary}
\mathbbm{1}_m = \begin{pmatrix}
	1\\
	1\\
	\vdots\\
	1\\
	\end{pmatrix}\in\R^m,\quad
\boldsymbol{Y}=\begin{pmatrix}
	Y_1\\
	Y_2\\
	\vdots\\
	Y_m\\
	\end{pmatrix}\in\R^{m \cdot d},\quad
 g[\boldsymbol{Y}] = \begin{pmatrix}
	g(Y_1)\\
	g(Y_2)\\
	\vdots\\
	g(Y_m)\\
	\end{pmatrix}\in\R^{m \cdot d}
 \end{equation}
         In order to implement a general $m$-stage  implicit  Runge-Kutta method, for example, an $m$-stage Radau method, one employs Newton's method or variants of it. These methods make use of the Jacobian of the system which is of size $m d\times md$. Diagonally implicit Runge-Kutta methods (DIRK) are an attractive alternative to reduce this computational cost. This is because in this case, the matrix $\calA$ in \eqref{eq:nonlinearimp} is lower triangular, which implies that the Jacobians used by the Newton method or its variants to solve \eqref{eq:nonlinearimp} are of dimension $d\times d$. Furthermore, when all the diagonal coefficients are equal, the method is referred to as singly diagonally implicit (SDIRK), which may offer the potential to further reduce computational costs \cite{HaW96}. A particular diagonal implicit method of interest in this paper is SDIRK4 as presented in \cite[Table IV.6.5, p. 100]{HaW96} and whose Butcher tableau is given by
         \begin{equation}\label{eq:sdirk4}
\begin{array}
{c|ccccc}
\\
 & \frac14\\[1mm]
 &\frac12 &\frac14 \\[1mm]
& \frac{17}{50}& \frac{-1}{25}&\frac14 \\[1mm]
& \frac{371}{1360} &\frac{-137}{2720} &\frac{15}{544} &\frac14\\[1mm]
&\frac{25}{24} &\frac{-49}{48} &\frac{125}{16} &\frac{-85}{12} &\frac14\\[1mm]
\hline
y_1&\frac{25}{24} &\frac{-49}{48} &\frac{125}{16} &\frac{-85}{12} &\frac14\\[1mm]
\hat y_1&\frac{59}{48} &\frac{-17}{96} &\frac{225}{32} &\frac{-85}{12} &0\\[1mm]
err&-\frac{3}{16} &\frac{-27}{32} &\frac{25}{32} & 0 &\frac14
\end{array}.
\end{equation}
\begin{remark}%
Note that for SDIRK4 we have $y_{n+1}=Y_m$, where $m=5$. We call such methods stiffly accurate \cite{HaW96}. 
{The coefficient $\gamma=1/4$ is chosen in \cite[Sect.\thinspace IV.6, p. 99]{HaW96} to achieve both $L$-stability and a favorable error constant with an embedded method for error control and a continuous output of order 3. In \cite[Sect.\thinspace IV.6, p. 99]{HaW96}, a numerical better choice with $\gamma=4/15$ and numerical values of the $L$-stable Runge-Kutta coefficients is also proposed but we choose in the paper the SDIRK4 method \eqref{eq:sdirk4} for simplicity due to its nice rational coefficients.}
\end{remark}
\begin{remark} \label{rem:variable_step}
{
The SDIRK4 method in \eqref{eq:sdirk4} also includes an embedded method $\hat y_1$ for a variable time-step strategy \cite{HaW96} based on the error estimator 
\begin{equation} \label{eq:shampine}
err=J^{-1}(y_1-\hat y_1)
\end{equation}
where each new step $y_1$ with time step $\Delta t$ is accepted if $err$ is small enough compared to the tolerance $tol$ or rejected otherwise. The next time step is then computed as $\Delta t_{new} = \xi \Delta t (tol/err)^{1/5}$ where $\xi=0.8$ is a safety factor and $tol$ is the tolerance prescribed by the user.
The term $J^{-1}$ in \eqref{eq:shampine} denotes the inverse of the Jacobian matrix $J=I_d-\Delta t \gamma \frac{\partial g}{\partial y}$ computed during the Newton method or its variant at iterates of the internal stages $Y_i$, and hence $J^{-1}$ is already computationally available with minor overhead in practice. This is a standard methodology for stiff problems proposed by Shampine, as presented in \cite[Sect.\thinspace IV.8]{HaW96}, to circumvent the lack of $L$-stability of the embedded method and hence  stabilize the error estimator which would otherwise be too pessimistic and fail for stiff problems. 
}
\end{remark}
       
        \subsection{Explicit stabilised Runge-Kutta methods} \label{sec:rkc}
        Even with the dimensionality reduction of diagonally implicit Runge-Kutta methods in terms of the Jacobians used for solving the nonlinear equations, implicit methods in high dimensions become prohibitively expensive. On the other hand, explicit schemes avoid the solution of such high-dimensional systems, but have bounded stability domains.  For this reason, standard explicit methods face severe step size restrictions when integrating stiff problems. Explicit stabilised Runge-Kutta methods represent a good compromise, they are fully explicit schemes with extended stability domains over the negative real axis.

        
        Explicit stabilised Runge-Kutta methods are constructed to have (nearly) maximal stability domain sizes along the negative real axis constrained by the fact that they remain explicit. In particular,  for an $s$-stage method the stability polynomial would be of degree $s$. The goal is then  to find, for a given integer $p$ (the order of the method), a polynomial of degree $s$ (the number of internal stages of the method) satisfying 
        \begin{equation*}
            R_s(z)=1+z+\frac{z^2}2+\dots+\frac{z^p}{p!}+\alpha_{p+1}z^{p+1}+\dots+\alpha_sz^s,
        \end{equation*}
        such that $|R_s(z)|\leq1$ for $z\in [-L_s,0]$ with $L_s>0$ is as large as possible. For $p=1$, the optimal value of $L_s$ is $2s^2$ \cite{Abdrev,HaW96}. The corresponding polynomial is the shifted Chebyshev polynomial $R_s(z)=T_s(1+z/s^2)$ where $T_s(\cos\theta) = \cos(s\theta)$ is the first kind Chebyshev polynomial of degree s. This polynomial oscillates between $-1$ and $1$ which introduces singularities in the stability domain and prevents convergence for some eigenvalues. For this reason, we modify the stability function in the following way
        \begin{equation}\label{eq:stabcheb}
            R_s(z)=\frac{T_s(\omega_0+\omega_1z)}{T_s(\omega_0)},
        \end{equation}
        where $\omega_0=1+\eta/s^2$, $\omega_1=T_s(\omega_0)/T_s'(\omega_0)$, and $\eta$ is called the damping parameter. The polynomial \eqref{eq:stabcheb} now oscillates between $-\alpha_s(\eta)$ and $\alpha_s(\eta)$ for every $z\in[-L_{s,\eta},-\delta_\eta]$, where
        \begin{equation}\label{eq:alphaL}
            \alpha_s(\eta)=\frac1{T_s(\omega_0)}<1, \quad L_{s,\eta}=\frac{1+\omega_0}{\omega_1},
        \end{equation}
        where $\delta_\eta$ is a positive real close to zero ($R_s(0)=1$). 

        Using the well-known three-term recurrence relation of the first kind Chebyshev polynomials $T_{j+1}(z)=2zT_j(z)-T_{j-1}(z)$, for $j\geq2$ (with $T_0(z)=1$ and $T_1(z)=z$), we can derive an efficient and stable implementation (in terms of memory and rounding errors) of the so-called Chebyshev method or the first order Runge-Kutta-Chebyshev (RKC) method with stability function \eqref{eq:stabcheb},  as analyzed by \cite{vanderHouwen1980} (also for second order explicit stabilised methods). When applied to non-linear problems where the Jacobian of $g$ has negative real eigenvalues, this method allows for taking much larger time steps than standard explicit methods would allow. Algorithm \ref{alg:rkc} summarises one step of the first order $s$-stage RKC method with step-size $\Delta t$ for the solution of \eqref{eq:ode}. 
\begin{algorithm}  
\caption{{First order} Runge-Kutta Chebyshev for the time integration of \eqref{eq:ode}} 
\textbf{Input:} Initial state $y_{n}$, damping $\eta$, upper bound $L$ for the eigenvalues of  the Jacobian of $g$ at $y_{n}$, time step $\Delta t$
        \begin{equation}\label{eq:rkc1}
            \begin{split}
            K_0 &= y_n,\quad K_1=K_0+\mu_1 \Delta t g(K_0),\\
            K_j &= \mu_j \Delta t g(K_{j-1})+\nu_jK_{j-1}+(1-\nu_j)K_{j-2},\quad j=2,\dots,s,\\
            y_{n+1} &= K_s,
            \end{split}
        \end{equation}
        where
        $$\mu_j=\frac{2\omega_1T_{j-1}(\omega_0)}{T_j(\omega_0)},\quad \nu_j=\frac{2\omega_0T_{j-1}(\omega_0)}{T_j(\omega_0)}$$
        and $s$ should be chosen large enough such that $L_{s,\eta}>L\Delta t$ 
        \newline
        \textbf{Output:} Solution $y_{n+1}$ after a time step of size $\Delta t$ 
        \label{alg:rkc}
\end{algorithm}
Note that there are three parameters needed for implementing this Algorithm, the damping $\eta$, the number of stages $s$ as well as the  step-size $\Delta t$. 
{For a fixed value of the damping parameter $\eta$ and a given time step $\Delta t$ one needs to choose the stage parameter $s$ in such a way that $L_{s,\eta} \geq L \Delta t$.
Typically for ODE integration one takes $\eta$ small (e.g $\eta=0.05$) which gives a nearly optimal stability domain size $L_{s,\eta} \geq (2-\frac{4}{3}\eta)s^{2}$, and a stage parameter
\[
s=\left\lceil \sqrt{\frac{L \Delta t}{2-\frac{4}{3}\eta}}\right\rceil
\] where $L$ is an upper bound for the largest eigenvalue of the Jacobian of $g$ evaluated at $y_{n}$.
For larger values of damping parameter $\eta$, the width $L_{s,\eta}$ of the stability domain reduces further but still grows quadratically with $s$ and the number of stages $s$ is then chosen large enough accordingly, see  for instance \cite[Remark 6.1]{AAV18} for details.
}
\subsection{Application to optimization}

In the recent paper \cite{EVVZ21} the authors utilised \eqref{eq:rkc1} to solve efficiently stiff optimization problems by an algorithm called  Runge-Kutta Chebyshev descent {(RKCD)}.  Here, by stiff optimization problems, we mean the case where the Hessian of the objective function has large eigenvalues distributed in the positive half-plane close to the real axis. In particular, consider the problem 
	\begin{equation}\label{eq:min}
		\min_{x\in\ \R^d}f(x),
	\end{equation}
	where the dimension $d$ is large and $f:\R^d\to\R$,  is $\ell$-strongly convex differentiable function, with $L$-Lipschitz derivative, and such that the real parts of the eigenvalues of its Hessian matrix $\mathcal{H}f(x)$ are strictly positive for all $x\in\R^d$. The condition number of the Hessian is bounded by $\kappa=\frac{L}{\ell}$.
 
 A starting point for solving the minimization problem \eqref{eq:min} is to consider discretizations of the gradient flow
	\begin{equation}\label{eq:gf}
		\dot x(t) = -\nabla f(x(t)), \qquad x(0)=x_0 \in \R^d,
	\end{equation}
	where $x_0$ is an initialization. Applying an explicit Euler method with step size\footnote{We will use $\Delta t$ for the actual time step of numerical integrators for differential equations and $h$ for optimization steps.} $h$ to \eqref{eq:gf} is equivalent to the gradient descent method
	\begin{equation}\label{eq:gradient descent}
		x_{k+1}=x_k-h\nabla f(x_k),\qquad k=0,1,2,...
	\end{equation}
	where $x_k$\footnote{When we describe an optimization algorithm to approximate a minimizer $x_{*}$ we will use the notation $x_{k}$ for the $k$-th iterate of the algorithm.} is the numerical approximation of $x(t_k)$. In the strongly convex setting described above when  ($L\gg l$), this method faces a severe restriction on $h$ to be stable. Furthermore, it is known that the optimal $h$ in terms of convergence speed towards the minimizer is given by {$h=\frac{2}{L+\ell}$ \cite{N18}}.  This implies that the number of iterations needed to find the minimizer of $f$ scales linearly with the condition number $\kappa$. Note that this behaviour is suboptimal since the most efficient optimization methods that make use only of the gradient of $f$ need $\mathcal{O}(\sqrt{\kappa})$ iterations to {find the minimum \cite{N18}}.  

To improve the convergence speed of the explicit Euler method one can apply a more sophisticated discretization of the gradient flow \eqref{eq:gf}. For example, one can use the first order RKC \eqref{eq:rkc1} to integrate  \eqref{eq:gf}. In particular, if one has estimates for the smallest and largest eigenvalues of the Hessian matrix $\nabla^{2} f$, then by choosing the parameters $\eta, h,$ and $s$ in \eqref{eq:rkc1} appropriately one can obtain a very efficient optimization method called RKCD \cite{EVVZ21} for strongly convex problems, see also Algorithm \ref{alg:rkcd}. In the case of RKCD, a calculation of $x_{k+1}$ for $x_{k}$ requires $s$ gradient evaluation. Now the following proposition from \cite{EVVZ21} characterizes the speed of convergence of iterates $x_{k}$ towards the minimizer  $x_{\star}$ when $f$ is quadratic in \eqref{eq:min}.

\begin{algorithm} [t]
\caption{{\cite{EVVZ21}} Runge-Kutta Chebyshev descent method (RKCD)} 

\textbf{Input:}
The gradient force $g=-\nabla f$ of a differentiable function $f:\mathbb{R}^d\rightarrow\mathbb{R}$.
Damping $\eta$. Lower and upper bounds $\ell$ and $L$ for eigenvalues of the Jacobian of $g$.  Initialisation $x_0\in\mathbb{R}^d$. 

\textbf{Body:} 
\begin{itemize}
  \item Compute the time step $h$ (learning rate) and the number of stages $s$ using 
  \begin{equation}  \label{eq:defn of parameters1}
  s=\left\lceil\sqrt{(\kappa-1)\eta/2}\right\rceil,\quad h=\frac{\omega_0-1}{\omega_1\ell}
  \end{equation}
  where $\kappa=L/\ell$  and
\begin{equation}\label{eq:defn of parameters}
\omega_{0}=1+\frac{\eta}{s^{2}}, \quad \omega_{1}
=\frac{T_{s}(\omega_{0})}{T'_{s}(\omega_{0})},
\end{equation}
with $T_s$ is the Chebyshev polynomial of the first kind with degree $s$. 
\item Until convergence, for $k\in \{0,1,2,3,\ldots\}$, \textbf{repeat}:
\begin{itemize}
\item Set $x_k^0 = x_k$ and $x_k^1 = x_k^0 + h\mu_1 g(x_k)$ with $\mu_1 = \omega_1/ \omega_0$
\item For $j\in \{2,\cdots,s\}$, \textbf{repeat}:
$$x_k^j = \mu_j h g(x_k^{j-1}) + \nu_j x_{k}^{j-1}-(\nu_j-1) x_{k}^{j-2},$$ 
\[
\mu_j = \frac{2\omega_1 T_{j-1}(\omega_0)}{T_j(\omega_0)},
\quad \nu_j = \frac{2\omega_0 T_{j-1}(\omega_0)}{T_j(\omega_0)}.
\]
\item Set $x_{k+1}=x_{k}^s$. 
\end{itemize}
\end{itemize}

\vspace{0.1cm}

\textbf{Output: } 
Estimate $\widehat{x}={x}_{n+1}$ of a minimiser of  ~\eqref{eq:min}. 

\label{alg:rkcd}
\end{algorithm}

 \begin{proposition} \label{prop:quad_RKC}
 {\cite{EVVZ21}.} {Let $f(x)=\frac12 x^TCx+b^Tx+c$ and consider the minimization problem \eqref{eq:min}}. Now let $x_{k}$ be the iterate of Algorithm \ref{alg:rkcd} with parameters chosen according to \eqref{eq:defn of parameters1} and
 \eqref{eq:defn of parameters}. 
Then the following {convergence estimates hold,}
 \begin{equation} \label{eq:ratte}
         \|x_{k}-x_*\|\leq \alpha^{k}_{s}(\eta)\|x_0-x_*\|,
     \end{equation}
     and
 \begin{equation}
         f(x_{k+1})-f(x_*)\leq \alpha^{2}_{s}(\eta)(f(x_k)-f(x_*)),
     \end{equation}
where  $x_{*}$ is the unique minimizer of $f$, and $\alpha_{s}(\eta)$ satisfies \eqref{eq:alphaL}. 

\end{proposition}

    A straightforward calculation shows that the learning rate $h$ in \eqref{eq:defn of parameters1} satisfies
    $
    h= \tfrac{\eta}{\ell} + \bigo(\eta^2 + s^{-2}),
    $
    for $\eta\rightarrow 0, s\rightarrow +\infty$ which makes it mildly dependent on $s$. 
Precisely, the following estimates holds.
\begin{lemma}
The learning rate $h$ defined in \eqref{eq:defn of parameters1} satisfies the bound
\begin{equation} \label{eq:h_bound}
h\leq \frac{\eta e^\eta}{\ell}.
   \end{equation}
\end{lemma}
\begin{proof}
We recall the following formula for all $s\geq 1,k\geq 0$,
\begin{equation} \label{eq:Tsprime}
T_s^{(k)}(1) = \prod_{j=0}^{k-1} \frac{s^2-j^2}{2j+1} \leq s^{2k}.
\end{equation} 
Using $1/\omega_1 = T_s'(\omega_0)/T_s(\omega_0)\leq T_s'(\omega_0)$, and performing a Taylor expansion of $T_s'(\omega_0)$ for $\eta$ close to 
$0$ yields
$$
\ell h=\frac{\eta}{s^2\omega_1} \leq \frac{\eta T_s'(\omega_0)}{s^2}
=\sum_{k=0}^{s-1} \frac{T_s^{(k+1)}(1)}{k! s^{2k+2} }{\eta^{k+1}} \leq 
\sum_{k=0}^{s-1} \frac{\eta^{k+1}}{k!} \leq \eta e^\eta
$$
where we used \eqref{eq:Tsprime} in the last inequality, which concludes the proof. 
\end{proof}
Using this proposition and the properties of $\alpha_{s}(\eta)$  it is possible to show that the effective rate of convergence of function iterates towards the minimizer $x_{*}$ scales like $1-c(\eta)/\sqrt{\kappa}$  when $\kappa \gg 1$ \cite{EVVZ21}, exhibiting thus accelerated behavior from an optimization point of view. This, in turn, implies that  to calculate $x_{\star}$ with target  accuracy $\epsilon$ one needs to do $\mathcal{O}(\sqrt{\kappa}\log{\epsilon^{-1}})$ gradient evaluations. 


 
        
	\section{A new explicit implementation of SDIRK  methods}
 \label{sec:main}
As discussed in Section \ref{subsec:hoirk}, implicit methods are usually implemented using Newton's method or its variants which can be memory demanding and delicate for ill-conditioned problems in particular for high-dimensional problems arising from discretizing PDEs.  Here, we propose an alternative way of implementing implicit Runge-Kutta methods 
\eqref{eq:nonlinRK}-\eqref{eq:nonlinRKfinal}
based on explicit stabilised methods. To do this, one needs to find the solution $\boldsymbol{Y}_{n+1}$ to the system of non-linear equations \eqref{eq:nonlinearimp}  given the current state $y_{n}$. Our starting point for achieving this is 
to introduce the auxiliary system of ODEs,
\begin{equation}\label{eq:gf1}
			\dot {\boldsymbol{Y}}(t)=-\boldsymbol{Y}(t)
	+\mathbbm{1}_m \otimes y_n
	+
	\Delta t (\calA \otimes I_d)
	g[\boldsymbol{Y}(t)]
\end{equation}  
that satisfies under appropriate assumptions 
\[
\lim_{t  \rightarrow +\infty} \boldsymbol{Y}(t) = \boldsymbol{Y}_{n+1}.
\]
We have thus reformulated the solution of an $m$-stage implicit Runge-Kutta method at time $t_{n+1}$ as the steady-state solution of the  ODE \eqref{eq:gf1}. 
In the simplest case of the implicit Euler method and assuming a gradient structure $g(y)=-\nabla V(y)$ in \eqref{eq:ode}, it is easy to see that  \eqref{eq:gf1} also has a gradient structure \eqref{eq:gf} with 
$ f(y)=\Delta t V(y)+\frac12 \| y - y_n\|_2^2. $ 
In this simple situation, following the ideas in \cite{EVVZ21} one can apply RKCD directly to compute $\boldsymbol{Y}_{n+1}$. 
However, this gradient structure assumption is not satisfied in general for a diagonally implicit Runge-Kutta method, even when \eqref{eq:main} has a gradient structure. In particular, \eqref{eq:gf1} in this case becomes
\begin{equation} \label{eq:gff}
\dot {\boldsymbol{Y}}(t)=-\boldsymbol{Y}(t) +\mathbbm{1}_m \otimes y_n
	+
	\Delta t (\calA \otimes I_d)(F_{D}[{\boldsymbol{Y}}(t)]+F_{A}[{\boldsymbol{Y}}(t)]+F_{R}[{\boldsymbol{Y}}(t)])
\end{equation}
As we will see, without this gradient assumption, applying Algorithm \ref{alg:rkcd}  to \eqref{eq:gff} would lead to arbitrarily slow convergence to $\boldsymbol{Y}_{n+1}$ as the stiffness grows. To tackle this issue, and inspired by the partitioning $\calA=\gamma I_m+(\calA-\gamma I_m)$ of the Runge-Kutta coefficients into diagonal part and strictly lower diagonal part proposed in \cite {VSC92},
we consider the following partition of \eqref{eq:gff},
\begin{subequations} \label{eq:gff_part}
\begin{align}
\dot{{\boldsymbol{Y}}}(t) \;=\;& \mathbbm{1}_m \otimes y_n - {\boldsymbol{Y}}(t) \;+\; \Delta t\, \gamma\, I_m \otimes I_d \, F_D[ {\boldsymbol{Y}}(t)] &&=:~G_D( {\boldsymbol{Y}}(t))\\
&\quad +~ \Delta t\, (\calA-\gamma I_m)\otimes I_d \, F_D[ {\boldsymbol{Y}}(t)] \;+\; \Delta t\, \calA\otimes I_d \, F_A[{\boldsymbol{Y}}(t)] &&=:~G_A( {\boldsymbol{Y}}(t))\\
&\quad +~ \Delta t\,\calA\otimes I_d \, F_R[ {\boldsymbol{Y}}(t)] &&=:~G_R( {\boldsymbol{Y}}(t)),
\end{align}
\end{subequations}
where we split the coefficient matrix $\mathcal{A}$ of the
singly diagonally implicit Runge-Kutta method
into its diagonal part $\gamma I_m$ and its strictly lower triangular part $\mathcal{A}-\gamma I_m$. 
The new Algorithm \ref{alg:3} is then defined as follows.
\begin{algorithm} 
\caption{{New} Runge-Kutta Chebyshev implementation of diagonally implicit Runge-Kutta methods \label{alg:3} }

\textbf{Input:}
The vector fields $\mathbb{R}^d\rightarrow\mathbb{R}^d$, $F_D$ (diffusion terms), $F_A$ (advection terms), $F_R$ (reaction terms).
The initial state $y_{n}\in\mathbb{R}^d$, an approximation ${\boldsymbol{x}}_0\in\mathbb{R}^{md}$ of ${\boldsymbol{Y}}_{n}$. Upper bounds $\lambda_{max}$ for the spectral radius of $\nabla F_D$.  
Diagonally Runge-Kutta coefficient matrix $A$ with value $\gamma>0$ on the diagonal, time-step size $\Delta t$. Damping $\eta$. 

\vspace{0.1cm}

\textbf{Body:} 
\begin{itemize}
  \item Compute the time step $h$ (learning rate) and the number of stages $s$ according to 
  \[
  s=\left\lceil\sqrt{(\kappa-1)\eta/2}\right\rceil,\quad h=\frac{\omega_0-1}{\omega_1\ell}
  \]
  where $\ell=1$, $L=1+\gamma \Delta t \lambda_{max},\kappa=L/\ell$ and
\[
\omega_{0}=1+\frac{\eta}{s^{2}}, \quad \omega_{1}
=\frac{T_{s}(\omega_{0})}{T'_{s}(\omega_{0})},
\]
with $T_s$ is the Chebyshev polynomial of the first kind with degree $s$. 
\item Until convergence, for $k\in \{0,1,2,3,\ldots\}$, \textbf{repeat}:
\begin{itemize}
\item Define 
\begin{align*}
G({\boldsymbol{x}}):&=G_D({\boldsymbol{x}}) - G_D({\boldsymbol{x}}_k)
+ J^{-1}\!\bigl(G_D({\boldsymbol{x}}_k)+G_A({\boldsymbol{x}}_k)+G_R({\boldsymbol{x}}_k)\bigr),
\\
J &= I_{md} - h\, G'_R({\boldsymbol{x}}_k),
\end{align*}
\item Set ${\boldsymbol{x}}_k^0 = {\boldsymbol{x}}_k$ and ${\boldsymbol{x}}_k^1 = {\boldsymbol{x}}_k^0+ h\mu_1 G({\boldsymbol{x}}_k)$ with $\mu_1 = \omega_1/ \omega_0$
\item For $j\in \{2,\cdots,s\}$, \textbf{repeat}:
$${\boldsymbol{x}}_k^j = \mu_j h G({\boldsymbol{x}}_k^{j-1}) + \nu_j {\boldsymbol{x}}_{k}^{j-1}-(\nu_j-1) {\boldsymbol{x}}_{k}^{j-2},$$ 
\begin{equation} \label{eq:rkcd_iter}
\mu_j = \frac{2\omega_1 T_{j-1}(\omega_0)}{T_j(\omega_0)},
\quad \nu_j = \frac{2\omega_0 T_{j-1}(\omega_0)}{T_j(\omega_0)}.
\end{equation}
\item Set ${\boldsymbol{x}}_{k+1}={\boldsymbol{x}}_{k}^s$. 
\end{itemize}
\end{itemize}

\vspace{0.1cm}

\textbf{Output: } 
Approximation ${\boldsymbol{x}}_{k+1}$ of the Runge-Kutta solution ${\boldsymbol{Y}}_{n+1}\in\mathbb{R}^{md}$ after one step. 

\label{alg:explicitimplicit}
\end{algorithm}

\begin{remark}
Algorithm \ref{alg:explicitimplicit} 
calculates
 $\boldsymbol{Y}_{n+1}$ by applying  Algorithm \ref{alg:rkcd}  to the following new partitioned auxiliary problem:
\begin{align} \label{eq:part}
\dot{\boldsymbol{Y}}
&= G_D(\boldsymbol{Y}) - G_D(\boldsymbol{Y}_n)
+ J^{-1}\!\bigl(G_D(\boldsymbol{Y}_n)+G_A(\boldsymbol{Y}_n)+G_R(\boldsymbol{Y}_n)\bigr),
\\
J &= I_{md} - h\, G'_R(\boldsymbol{Y}_n).
\end{align}
Note that this partitioning preserves the steady state ${\boldsymbol{Y}}_{n+1}$ of \eqref{eq:gff} since 
\[
G_{D}({\boldsymbol{Y}}_{n+1})+G_{A}({\boldsymbol{Y}}_{n+1})+G_{R}({\boldsymbol{Y}}_{n+1})=0.
\]
\end{remark}

\begin{remark}
The partitioning \eqref{eq:gff_part} and the corresponding treatment of the $G_{R}$ term in \eqref{eq:part} assume that the reaction term $F_{R}$ is stiff. Note that the Jacobian inversion in this setting is cheap as it has a block diagonal structure thanks to the local nature of the reaction terms.  In practice, for a non-stiff reaction, one takes $G_{R}$ to be zero to avoid the Jacobian computations and then take \[G_{A}(\boldsymbol{Y}(t)):=\Delta t\, (\calA-\gamma I_m)\otimes I_d \, F_D[ {\boldsymbol{Y}}(t)] \;+\; \Delta t\, \calA\otimes I_d \, \left(F_A[{\boldsymbol{Y}}(t)]+ F_R[ {\boldsymbol{Y}}(t)]\right).
\]
This corresponds to treating together the advection and non-stiff reaction terms as first proposed for PRKC \cite{Zb11} and also used in PIROCK \cite{AV13}.
\end{remark}

In Algorithm \ref{alg:3}, we set $\ell=1$ for simplicity, as the spectral gap for the Jacobian of $G_D$. Note that if $\lambda_{\min}\geq 0$ denotes the smallest eigenvalue of the Jacobian of $-F_D$ one could use the more precise value $\ell=1+\Delta t\lambda_{\min}$. 
Concerning the choice of the damping parameter $\eta$, while the analysis applies for any value $\eta>0$, it was shown numerically in \cite{EVVZ21} that the value $\eta=1.17$ is an efficient choice for RKCD. For Algorithm \ref{alg:3}, we choose $\eta=4$ which reveals efficient in the considered numerical test problems (see Section \ref{sec:num}). Concerning the stopping criteria for the main iteration of Algorithm \ref{alg:3}, we use for simplicity the condition $\|\boldsymbol{x}_k-\boldsymbol{x}_{k-1}\|<tol_{it}$, with $tol_{it}=10^{-12}$. For the initialisation, we use for simplicity $\boldsymbol{x}_0=\mathbbm{1}_m\otimes y_n$ where $y_n$ is the current state.


\section{Convergence and stability analysis} \label{sec:conv}
In this section, we prove the convergence of Algorithm \ref{alg:3} in two special cases, pure linear diffusion and linear advection-diffusion. Additionally, we provide the rationale behind the partition in equation \eqref{eq:gff_part} for a general advection-diffusion-reaction equation.

\subsection{The case of pure linear diffusion}
We  consider 
\begin{equation}\label{eq:heat}
\partial_t u=a\Delta u, \quad (t,x)\in[0,+\infty)\times(0,1)^{N}
\end{equation}
with periodic boundary conditions and a smooth initial condition $u(x,0)=1+e^{-50\norm{x-c}^2}$ where $c$ is the center of the spatial domain.  We  now  discretise in space using a finite difference method with a uniform mesh size $\Delta x$ in each spatial dimension.  We thus have 
\begin{equation} \label{eq:pure_diffusion}
F_{D}(y)= -Dy, \quad F_{A}(y)=0 \quad F_{R}(y)=0
\end{equation}
in \eqref{eq:main} with  $D \in \mathbb{R}^{d\times d}$ is the discrete Laplacian with periodic boundary conditions. In this case, we have that the bound for the spectral radius of $\nabla F_{D}(y)=-D$ is given by $\lambda_{\text{heat},max}=4 N a /\Delta x^{2}$, for $N=1,2,3$.

\paragraph{Implicit Euler}
We consider the solution of \eqref{eq:main} by implicit Euler when the different terms are given by \eqref{eq:pure_diffusion}. In this case, we have that $m=1$ and $\mathcal{A}=\gamma=1$ in \eqref{eq:gff} implying that 
\begin{equation} \label{pure_diffusion1}
G_{D}(y)= \mathbbm{1}_m\otimes y_{n}-y-\Delta t D y, \quad G_{A}(y)=0 \quad G_{R}(y)=0
\end{equation}
in \eqref{eq:gff_part}.  
We now use Algorithm \ref{alg:explicitimplicit} to calculate the solution of implicit Euler at time $t_{n+1}$ given the solution $y_{n}$ at time $t_{n}$. Note that in this case, because $G_{A}(y)=G_{R}(y)=0$  and $m=1$ \eqref{eq:gff} has a gradient structure and Algorithm \ref{alg:3} coincides with Algorithm \ref{alg:rkcd}. One would thus expect from Proposition \ref{prop:quad_RKC} that given $y_{n}$, the cost to approximate $y_{n+1}$ scales proportionally to the 
$\sqrt{\kappa}$ where $\kappa$ is the condition number of $G_{D}$ given by $\kappa=1+\gamma \Delta t \lambda_{\text{heat},max}=1+4Na/\Delta x^{2}$.  To illustrate this, we now vary $\kappa$ by changing the size of the mesh refinement $\Delta x= 1/4,1/8,1/16,1/32,1/64$ while keeping $\Delta t=5 \times 10^{-2}$. The results of this experiment are presented in Figure \ref{fig:3dLaplace} for $N=1$ and $N=3$ spatial dimensions, and show that the cost in terms of gradient evaluations indeed scales like $\mathcal{O}(\sqrt{\kappa})$ independently of the value of $N$. Here we have used a prescribed tolerance $\varepsilon=10^{-10}$ in the criteria for stopping the iterations.


\begin{figure}
    \centering
    \begin{subfigure}{0.49\linewidth}
    \centering
    \includegraphics[width=\textwidth]{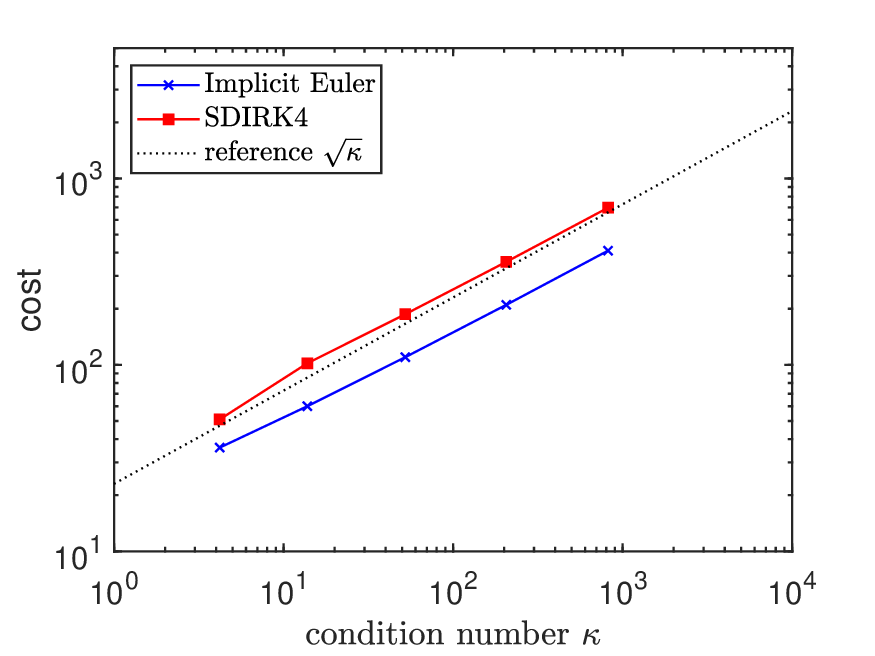}
    \caption{1D case, DOF: $4,8,16,32,64$}
    \end{subfigure}
    \begin{subfigure}{0.49\linewidth}
    \centering
    \includegraphics[width=\textwidth]{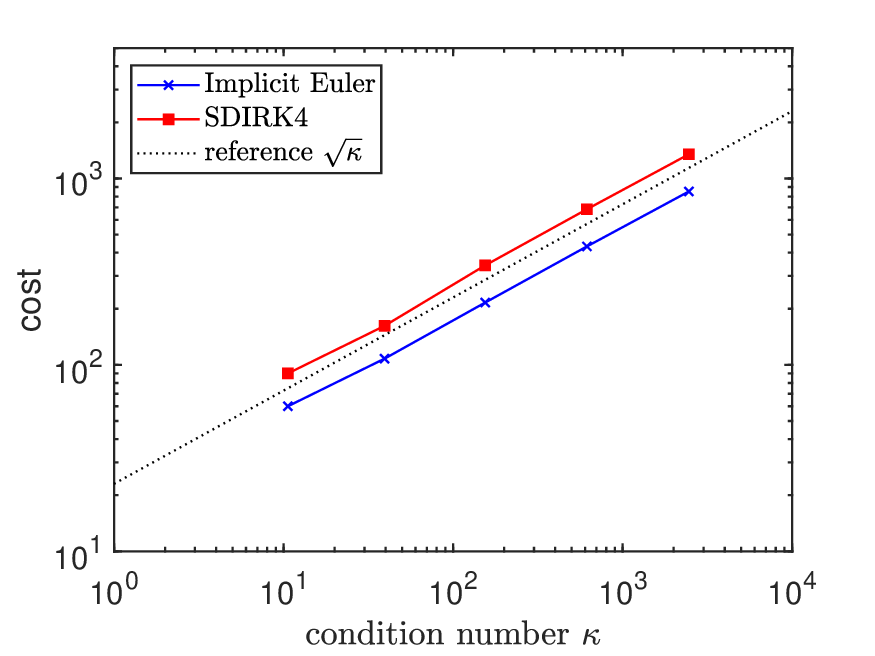}
    \caption{3D case, DOF: $4^3,8^3,16^3,32^3,64^3$}
    \end{subfigure}
    \caption{Comparison of the cost (number of matrix-vector products)   as a function of the condition number $\kappa$ for Algorithm \ref{alg:3}. }
    \label{fig:3dLaplace}
\end{figure}
\paragraph{SDIRK4}
We  now consider the solution of \eqref{eq:main} by SDIRK4 when the different terms are given by \eqref{eq:pure_diffusion}. In this case $m=5$ and the Runge-Kutta matrix $\mathcal{A}$ is given by \eqref{eq:sdirk4}, hence $\gamma=1/4$ and
{
\begin{subequations} \label{eq:pure_diffusion_sdirk}
\begin{align} 
G_{D}(\mathbf{Y}) &= \mathbbm{1}_{m}\otimes y_{n}-\mathbf{Y}-\frac{1}{4}\Delta t (I_m \otimes D) \mathbf{Y}, \\ G_{A}(\mathbf{Y})&= \Delta t \left[\left(\mathcal{A}-\frac{1}{4} I_{m} \right)\otimes D \right]\mathbf{Y}   \\ G_{R}(\mathbf{Y})&=0.
\end{align}
\end{subequations}
}

Unlike the implicit Euler method, we see that in the case of SDIRK4 
Algorithm \ref{alg:explicitimplicit}  doesn't coincide with Algorithm 
\ref{alg:rkcd}, hence Proposition \eqref{prop:quad_RKC} does not 
directly apply. Nevertheless,  as we observe\footnote{we have used the 
same parameters $\varepsilon, \Delta x, \Delta t$ as for the implicit 
Euler} in  Figure \ref{fig:3dLaplace} given $y_{n}$, the cost to 
approximate $y_{n+1}$ with Algorithm \ref{alg:3} scales again 
proportionally to $\mathcal{O}(\sqrt{\kappa})$  independently of the 
value of $N$, where $\kappa$ is the condition number of $G_{D}$. We will prove in Theorem \ref{thm:diff} that this is indeed a property of Algorithm \ref{alg:3} for the implementation of an $m$-stage SDIRK method in the pure linear diffusion case, equation \eqref{eq:pure_diffusion}.

\subsubsection{Why the partitioning?}
We now discuss in detail why the partitioning \eqref{eq:part} is necessary for the construction of   Algorithm \ref{alg:3} and why applying  Algorithm \ref{alg:rkcd} without the gradient assumption to \eqref{eq:gff} would fail in general, even in the simplest possible case of pure linear diffusion where $F_{D}, F_{R}, F_{A}$ are given by \eqref{eq:pure_diffusion}.  More precisely, then  \eqref{eq:gff} becomes
\begin{equation}\label{eq:naive_sdirk4}
\dot{ \boldsymbol{Y}}(t)= \mathbbm{1}_{m} \otimes y_{n} +\mathcal{V} \boldsymbol{Y}(t), 
\quad 
\mathcal{V}\coloneqq-I_{md}-\Delta t (\calA \otimes D).
\end{equation}
Note that because $\calA$ is lower triangular with all entries in the diagonal equal to $\gamma$ when $m \geq 2$ neither $\calA$ nor $\mathcal{V}$ can be put 
in diagonal form. In addition, the eigenvalues of $\mathcal{V}$ all have multiplicity $m$ and are given by 
\[
\lambda_{j}=-1-\Delta t \gamma \lambda_{\text{heat},j}  
\]
where $\lambda_{\text{heat},j}$ are the eigenvalues of $D$, and the corresponding eigenspace associated to $\lambda_{j}$ has dimension $1$. 

We start by considering the case of Implicit Euler ($\gamma=1$, $m=1$)  for $\Delta t=5\times 10^{-2}, N=1$. Since now $G_{A}=G_{R}=0$ as discussed above Algorithm \ref{alg:3}  coincides with Algorithm \ref{alg:rkcd} and thus  Proposition \ref{prop:quad_RKC} applies. This is indeed the case up to round-off errors as we can see in Figure \ref{fig:iterations} where we plot the error between the iterations $x_{k}$ for Algorithm \ref{alg:3} and  the exact solution $\boldsymbol{Y}_{n+1}$ calculated using standard Newton iteration.  

If we now consider  the case of SDIRK4 ($\gamma=1/4$ and $m=5$), $G_{A} \neq 0$ and then \eqref{eq:gff} even though linear in $\boldsymbol{Y}$ is not of gradient form. As we can  see in  Figure \ref{fig:iterations}, and we  also prove in Theorem \ref{thm:diff}, the partitioning \eqref{eq:part} is essential for ensuring the convergence of iterates, since the iterates calculated with the non-partitioned Algorithm \ref{alg:rkcd} {(corresponding to Algorithm \ref{alg:3} with 
the definition of $G(\boldsymbol{x})$ replaced by $G(\boldsymbol{x})=G_D(\boldsymbol{x})+G_A(\boldsymbol{x})$ and with $G_R(\boldsymbol{x})=0$, or equivalently applying Algorithm \ref{alg:rkcd} to equation \eqref{eq:gff} with $F_A=F_R=0$)} suffer from   high-round off errors  that   deteriorate as stiffness increases {(see  Appendix \ref{appendix:naive} for details).} In contrast, iterates calculated with Algorithm \ref{alg:3} do not suffer from these issues. 
 \begin{figure}
    \centering
    \begin{subfigure}{0.48\linewidth}
    \centering
\includegraphics[width=\textwidth]{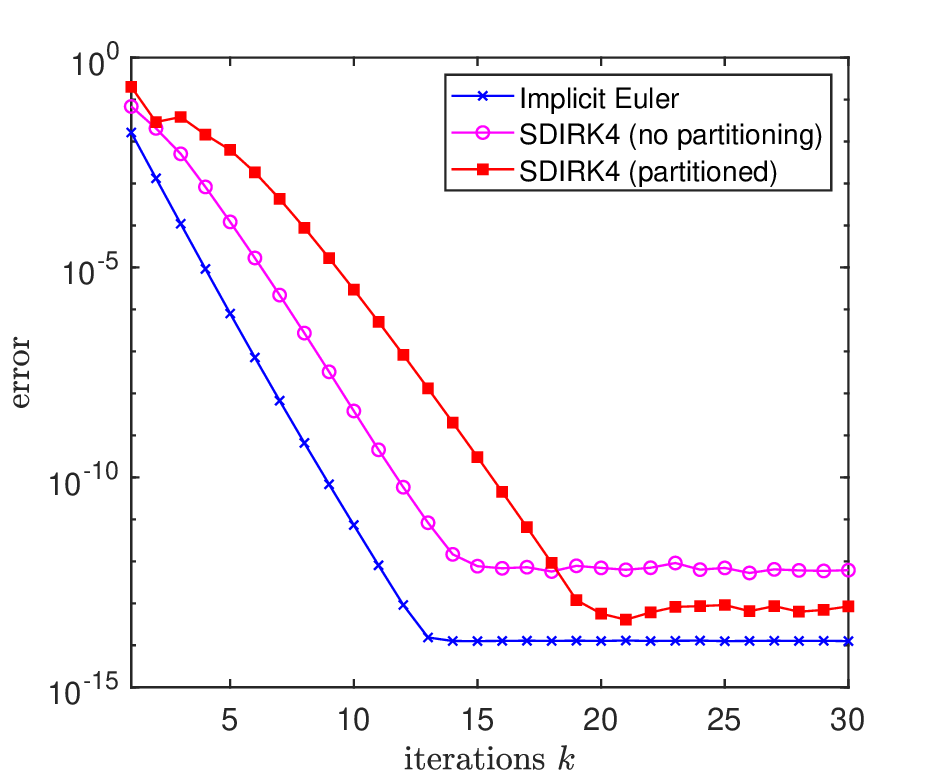}
    \caption{$\Delta x= 10^{-2},\kappa=0.2\cdot 10^{4}$.}
    \end{subfigure}
    \begin{subfigure}{0.48\linewidth}
    \centering
    \includegraphics[width=\textwidth]{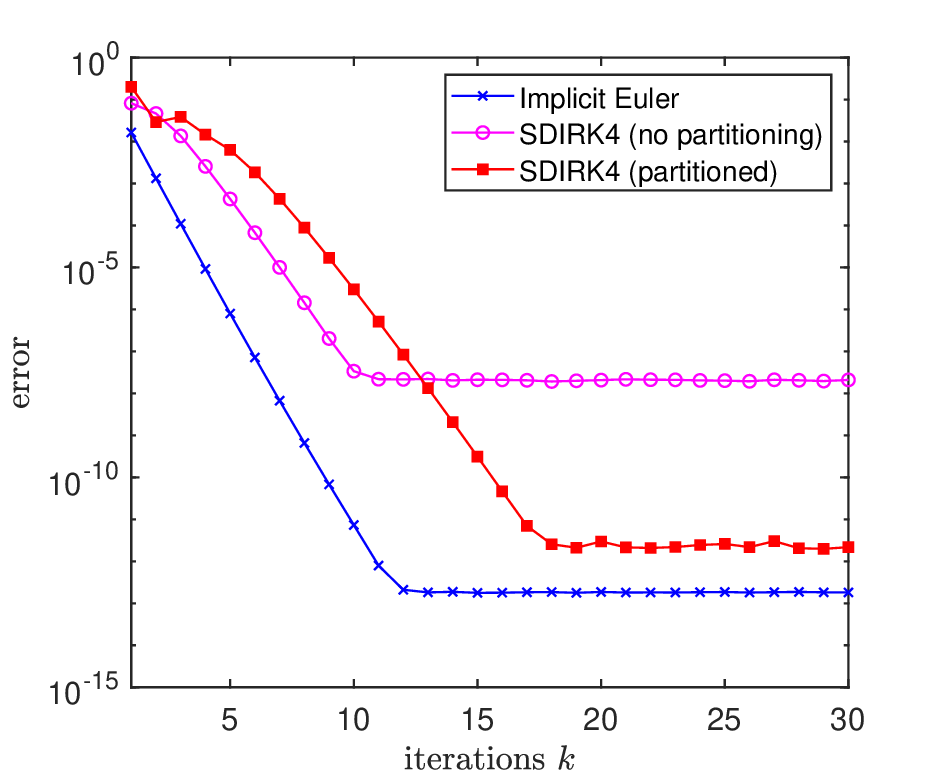}
    \caption{$\Delta x= 10^{-3},\kappa=0.2\cdot 10^{6}$.}
    \end{subfigure}
    \caption{Behaviour of the error  for Algorithm \ref{alg:3} and the non-partitioned method}
    
    \label{fig:iterations}
\end{figure}

We  explain this difference in depth in Section \ref{appendix:naive}, but in a nutshell, the issue is caused by the fact that the matrix $\mathcal{V}$ in \eqref{eq:naive_sdirk4}  is non-symmetric when $m \geq 2$  which creates issues for the non-partitioned method.  In the case of the partitioning \eqref{eq:part}, 
we obtain
\[
G_{D}(\boldsymbol{Y})=\mathbbm{1}_m \otimes y_n +\mathcal{V}_{1}\boldsymbol{Y},\qquad G_{A}(\boldsymbol{Y})=\mathcal{V}_{2}\boldsymbol{Y},\qquad
G_{R}(\boldsymbol{Y}) =0,
\]
where  the Jacobian $\mathcal{V}$ is decomposed in two parts $\mathcal{V}=\mathcal{V}_{1}+\mathcal{V}_{2}$ as
\begin{eqnarray*}
    \mathcal{V}_{1} &=& -I_{m}\otimes \hat{\mathcal{V}}_{1}, \quad \hat{\mathcal{V}_{1}}= (I_{d}+\gamma \Delta t D), \\
    \mathcal{V}_{2} &=&-\left(\mathcal{A}-\gamma I_{m} \right)\otimes \hat{\mathcal{V}}_{2}, \quad \hat{\mathcal{V}}_{2}=-\Delta t D.
\end{eqnarray*}
By construction, $\mathcal{V}_{1}$ is symmetric and has the same eigenvalues as $\mathcal{V}$, while $\mathcal{V}_{2}$ is non-symmetric and nilpotent. 

The iterates $\boldsymbol{x}_{k}$ for Algorithm \ref{alg:3} satisfy the following recurrence relation
\begin{eqnarray} 
    \boldsymbol{x}_{k+1}
   & = &
    R_s(h \mathcal{V}_{1})
    \boldsymbol{x}_k
   +
   hB_s(h \mathcal{V}_1)
   (\mathbbm{1}_m \otimes y_n +
	\mathcal{V}_2 \boldsymbol{x}_k) \label{eq:it_rkcd} \nonumber \\
&=&
    \boldsymbol{x}_k+ h B_s(h \mathcal{V}_{1})
   ( \mathbbm{1}_m \otimes y_n +
	\mathcal{V}\boldsymbol{x}_{k}) 
    \label{eq:it_rkcd1}
\end{eqnarray}
where $B_s(z)=(R_s(z)-1)/z$, and we used $R_s(z)=1+B_s(z)z$. It is important to note that the steady state of \eqref{eq:it_rkcd1} coincides with the steady state of \eqref{eq:naive_sdirk4} as desired. Furthermore, as we can see in \eqref{eq:it_rkcd}, only the symmetric part of the matrix {$\hat{\mathcal{V}}$} is taken as an argument of the stability function $R_{s}$, while the non-diagonalizable part {$\hat{\mathcal{V}}_2$} is stabilised by $B_s$. This is, in fact, key for establishing the convergence properties of Algorithm \ref{alg:3} as detailed in the next section. 
\begin{remark}
Another benefit of the splitting $\mathcal{V}=\mathcal{V}_1+\mathcal{V}_2$ is that it allows for a parallel implementation of the stabilization terms $ R_s(h \mathcal{V}_{1})$ and $B_s(h \mathcal{V}_{1})$ over $m$ processors due to their block diagonal structures. Note that a similar splitting was proposed in \cite{VSC92} with the motivation of parallelizing the  iterates of a quasi-Newton implementation of SDIRK methods.
\end{remark}


\subsubsection{Theoretical analysis of Algorithm \ref{alg:explicitimplicit}} \label{subsubsec:analysis}

\begin{theorem}\label{thm:diff}
Consider Algorithm \ref{alg:explicitimplicit} for the implementation of an $m$-stage SDIRK  method applied to ODE \eqref{eq:main}  where  $F_{D}, F_{A}, F_{R}$ are given by \eqref{eq:pure_diffusion}.
Then, for all $\varepsilon>0$ there exists $C_\varepsilon$ (independent of the dimension $d$) such that the iterates $\boldsymbol{x}_k$ of Algorithm \ref{alg:explicitimplicit} converge to the
SDIRK solution $\boldsymbol{Y}_{n+1}$ of \eqref{eq:nonlinearimp} with the following convergence estimate,
\begin{equation} \label{eq:proof1}
\|
\boldsymbol{x}_k - \boldsymbol{Y}_{n+1}
\|_2 \leq C_\varepsilon \big(\alpha_{s}(\eta) + \varepsilon\big)^k 
\|
\boldsymbol{x}_0 - \boldsymbol{Y}_{n+1}
\|_2.
\end{equation}
\end{theorem}
\begin{proof}
Using \eqref{eq:it_rkcd}, we obtain,
\begin{equation}
\boldsymbol{x}_{k+1} - \boldsymbol{Y}_{n+1}= M (\boldsymbol{x}_{k}- \boldsymbol{Y}_{n+1}),\qquad M=  R_{s}(h\mathcal{V}_1) + B_{s}(h\mathcal{V}_1)h\mathcal{V}_2,
\end{equation}
where   $B_{s}(z)=(R_{s}(z)-1)/z$. In order to prove \eqref{eq:proof1} we need to have an estimate of $\norm{M}$. The main difficulty in doing so is the term $\mathcal{V}_{2}$, as this is not diagonalizable.  To address this issue, we construct an appropriate change of coordinates which makes this term negligible. Precisely, we consider the Jordan decomposition of the Runge-Kutta matrix $\calA$ of coefficients,
$
\calA = P\mathcal{J}P^{-1}
$
where $P\in\R^{m\times m}$ is an invertible matrix, and $\mathcal{J}\in\R^{m\times m}$ is a Jordan block matrix of the form
\begin{equation}\label{eq:defJ}
\mathcal{J}= 
			\begin{pmatrix}
				\gamma & 1 & 0& \cdots & 0\\
				0& \gamma & 1& \cdots & 0\\
				\vdots&\vdots&\ddots&\\
				0& 0 & & \ddots & 1 \\
				0& 0 &  &\cdots & \gamma \\
			\end{pmatrix} =\gamma I_m + F. 
\end{equation}
Here, we assume for simplicity of the presentation that there is a single Jordan block $\mathcal{J}$, but the proof generalises straightforwardly to the case of multiple Jordan blocks.

Now let $Q_\varepsilon=\mathrm{diag}(1,\varepsilon,\varepsilon^2,\ldots, \varepsilon^{m-1})$.
Using 
		$
		Q_\varepsilon^{-1} F Q_\varepsilon = \varepsilon F
		$,
we deduce
$
Q_\varepsilon^{-1}P^{-1} \calA PQ_\varepsilon =  \gamma \mathbbm{1}_m +  \varepsilon F
$
and
$$
(Q_\varepsilon^{-1}P^{-1}\otimes I_d) M (PQ_\varepsilon\otimes I_d) =  R_s(h\mathcal{V}_1) - \varepsilon   B_s(h\mathcal{V}_1) [F \otimes h\hat{\mathcal{V}}_{2}],
$$
where we have used that $(A \otimes B)(C \otimes D)=AC \otimes BD$ and that the matrix $PQ_\varepsilon\otimes I_d$ and its inverse both commute with $R(h\mathcal{V}_1)=I_m \otimes R_s(h\hat{\mathcal{V}}_{1})$ and $B_s(\mathcal{V}_1)$. This yields
\begin{eqnarray}
\|\boldsymbol{x}_k-\boldsymbol{Y}_{n+1}\|_2 &=& \|M^k(\boldsymbol{x}_0-\boldsymbol{Y}_{n+1})\|_2 \nonumber \\
&\leq& \kappa_2(PQ_\varepsilon)
(\|R(h\hat{\mathcal{V}}_{1})\|_2 + \varepsilon \|F\|_2 \|B_{s}(h\hat{\mathcal{V}}_{1})h\hat{\mathcal{V}}_{2}\|_2)^k \|\boldsymbol{x}_0-\boldsymbol{Y}_{n+1}\|_2
\end{eqnarray}
where $\kappa_2(PQ_\varepsilon)=\|PQ_\varepsilon\|_2\|(PQ_\varepsilon)^{-1}\|_2$ is the condition number of the matrix $PQ_\varepsilon\in\R^{m\times m}$, which is independent of $d$.
{ We now have (see \cite[Lemma A.1]{EVVZ21})
\begin{equation} \label{eq:RBone}
\|R_{s}(h\hat{\mathcal{V}}_1)\|_2 \leq \alpha_s(\eta)<1,\qquad \|B_{s}(h\hat{\mathcal{V}}_1)\|_2 \leq 1, 
\end{equation}
In addition, we have   $\|F\|_2 = 1$, 
 $\hat{\mathcal{V}}_2=\gamma^{-1}(I_d -\hat{\mathcal{V}}_1)$,
and $h\leq\eta e^\eta/\ell$ using \eqref{eq:h_bound},
which yields $\|B(h\hat{\mathcal{V}}_1)h\hat{\mathcal{V}}_1\|_2 = \|R_s(h\hat{\mathcal{V}}_1) -I_d\|_2 \leq 2$, and 
\begin{equation} \label{eq:us_bound}
\|B_s(h\hat{\mathcal{V}}_1)h \hat{\mathcal{V}}_2\|_2 \leq  \|h \hat{\mathcal{V}}_2\|_2 \leq\gamma^{-1}(2  + \eta e^\eta/\ell )
\end{equation}
}
and hence
\begin{eqnarray*}
\|\boldsymbol{x}_k-\boldsymbol{Y}_{n+1}\|_2 &\leq& \kappa_2(PQ_\varepsilon)
(\alpha_s(\eta) + \varepsilon \gamma^{-1}(2  + \eta e^\eta/\ell ))^k \|\boldsymbol{x}_0-\boldsymbol{Y}_{n+1}\|_2 \\
&\leq&C_\varepsilon
(\alpha_s(\eta) + \varepsilon)^k \|\boldsymbol{x}_0-\boldsymbol{Y}_{n+1}\|_2 
\end{eqnarray*}
where $C_\varepsilon = \kappa_2(PQ_{\varepsilon'})$,
with
$\varepsilon' = \varepsilon\gamma(2  + \eta e^\eta/\ell )^{-1} $ is a constant independent of the dimension $d$
and this concludes the proof.
\end{proof}
{In the proof of Theorem \ref{thm:diff}, observe that one can improve the estimate \eqref{eq:us_bound} for the quantity $B(h\hat{\mathcal{V}_1})h \hat{\mathcal{V}}_2$}, but this is not useful here because $\varepsilon$ can be chosen arbitrarily small. 
\begin{remark}
In the case of the implicit Euler method as we discussed before Algorithm \ref{alg:3} is equivalent to Algorithm \ref{alg:rkcd}. This is reflected in the statement of Theorem \ref{thm:diff} since $\varepsilon$=0, $C_{\varepsilon}=1$ in \eqref{eq:proof1} which agrees with \eqref{eq:ratte}  in Proposition \ref{prop:quad_RKC}.    
\end{remark}

\subsection{The case of linear  advection-diffusion}

We now consider  the 
following  linear advection-diffusion equation
\begin{equation} \label{eq:advdiffexact}
\partial_t u=a\Delta u+b \cdot \nabla u, \quad (t,x)\in [0,+\infty)\times (0,1)^{N}
\end{equation}
with periodic boundary conditions and a smooth initial condition $u(x,0)$, where $a>0$ and $b\in \R^N_+$ are fixed.  We  now  discretise in space using a finite difference method with a uniform mesh size $\Delta x$ in each spatial dimension.  We thus have 
\begin{equation} \label{eq:adv_diff}
F_{D}(y)= -Dy, \quad F_{A}(y)=E y \quad F_{R}(y)=0.
\end{equation}
where $D$ is the discrete Laplacian with periodic boundary conditions and $E$ is the discrete centred advection operator.
Following the presentation in \cite{Zb11}  the eigenvalues $\lambda_{\text{heat-adv},j}=\lambda_{\text{heat},j}+i\mu_{\text{adv},j}$ of $F=F_D+F_A$ are located  on the following ellipse in the complex plane
\begin{equation} \label{eq:ellipse_pq}
\left(\frac{2p}\alpha+1\right)^2 + \left(\frac{q}\beta\right)^2=1,
\end{equation}
with half-height $\beta=\norm{b}_{1}\Delta x^{-1}$ and width $\alpha=4aN\Delta x^{-2}$, and where
$p=\lambda_{\text{heat},j}\in[-\alpha,0]$ located on the negative real axis are the eigenvalues of $F_D$ and $iq=i\mu_{\text{adv},j}\in[-i\beta,i\beta]$ on the imaginary axis are the eigenvalues of $F_A$.

Before we present a theorem for an $m$-stage SDIRK method it is useful to introduce the following function
\begin{equation}\label{eq:stabRIE}
\mathcal{R}_{s}(P,Q) =  R_s(P) + B_s(P)Q.
\end{equation}
that will play an important role in establishing the stability and convergence properties in the general case. For simplicity, we start by considering the implicit Euler case for which the iteration error for Algorithm \ref{alg:3} satisfies 
\begin{equation}\label{eq:iterrIE}
x_{k+1}-y_{n+1} = \mathcal{R}_{s}(-hI_d - h\Delta t D,hE)(x_k-y_{n+1})
\end{equation}
and consider the related stability domain 
\begin{equation} \label{eq:stab_domain}
{\mathcal{S}_\gamma} = \{(p,q) \in \mathbb{R}^2\ ;\ R_{\text{stab}}({\gamma} p,q) <1 \}.
\end{equation}
{with $\gamma=1$ for the implicit Euler method and} where
\[
R_{\text{stab}}(p,q)\coloneqq |R_{s}(-h+h p)|^2 + |B_{s}(-h+h p)|^2|h q|^2
\]
We first prove the stability property of Algorithm \ref{alg:3} in the simplest case of the implicit Euler method and based on the stability domain \eqref{eq:stab_domain}. We then show that it also implies the convergence property of Algorithm \ref{alg:3} for an $m$-stage SDIRK method. 

\begin{proposition}\label{prop:rate1}
The error $\norm{x_{k+1}-y_{n+1}}_{2}$ for the implicit Euler method given by \eqref{eq:iterrIE} converges to zero as $k\rightarrow+\infty$ if for all $p_{j}=-\Delta t \lambda_{\text{heat},j}, \ q_{j}=\Delta t\mu_{adv,j}, j=1,\ldots,d$, we have $(p_{j},q_{j}) \in \mathcal{S}_1$. In this case, we have
$$
\norm{x_{k+1}-y_{n+1}}_{2} \leq \rho\norm{x_{k}-y_{n+1}}_{2},\qquad \rho := \max_{j=1,\cdots,d} R_{\text{stab}}(p_j,q_j)^{1/2}<1.
$$
\end{proposition}
\begin{proof}
If $p_{j},q_{j} \in \mathcal{S}_1$ for all $j$ we have that
 $R_{\text{stab}}(p_{j},q_{j})<1$. Now, since we assume periodic boundary conditions, the matrices   $D,E$ commute which implies that  
$ 
 \|\mathcal{R}_{s}(-hI_d - h\Delta t D,hE) \|_2 = \rho
 <1,
$
which yields the convergence for the error $\norm{x_{n+1}-y_{n+1}}_{2}$ in \eqref{eq:iterrIE}.     
\end{proof}

In proving Proposition \ref{prop:rate1}, we have explicitly taken into account that because of the periodic boundary conditions, the matrices $D$ and  $E$ are commuting. This will, in general, not hold, for example in the case of Neumann or Dirichlet boundary conditions. It is still possible to establish a stability criterion, as stated in the next proposition. 

\begin{proposition} \label{prop:rate2}
Consider \eqref{eq:main} with $F_D,F_A,F_R$ given  in \eqref{eq:adv_diff} and where $D$ is a given positive semi-definite matrix and $E$ is skew symmetric matrix with bound $\norm{E} \leq \norm{b}_{1}\Delta x^{-1}$. Then the  error $\norm{x_{k+1}-y_{n+1}}_{2}$ for Algorithm \ref{alg:3} applied with the implicit Euler method given by \eqref{eq:iterrIE} converges to zero as $k\rightarrow+\infty$ if 
 \begin{equation} \label{eq:rate_estimate}
 \hat \rho := \alpha_s(\eta) + \frac{\eta e^\eta}{\ell}\norm{b}_{1}\frac{\Delta t}{\Delta x} < 1.
 \end{equation}
In this case we have
\[
\norm{x_{k+1}-y_{n+1}}_{2} \leq \hat \rho\norm{x_{k}-y_{n+1}}_{2}.
\]
\end{proposition}
\begin{proof} 
Using \eqref{eq:RBone}
and $\norm{E} \leq \norm{b}_{1}\Delta x^{-1}$, 
we obtain from \eqref{eq:iterrIE}, $ \|\mathcal{R}_{s}(-hI_d - h\Delta t D,hE) \|_2\leq  \alpha_s(\eta) + h \Delta t \norm{E} \leq \hat  \rho <1$ where we use  \eqref{eq:h_bound},\eqref{eq:rate_estimate} and the convergence holds.
\end{proof}

We now state the general result, that the estimate for implicit Euler generalises for an arbitrary $m$-diagonal SDIRK method.


\begin{theorem}\label{thm:ad}
Consider Algorithm \ref{alg:3} for the implementation of an $m$-stage SDIRK Runge-Kutta method applied to ODE \eqref{eq:main} where $F_{D},F_{A},F_{R}$ are given by \eqref{eq:adv_diff}. Assume that $D$ and $E$ satisfy either the assumptions of Proposition \ref{prop:rate2} or the stronger assumptions of Proposition \ref{prop:rate1}, and
{assume that at least one of the following bounds hold}
{\begin{align*}
\hat \rho:&= \alpha_s(\eta)+ \frac{\eta e^\eta}{\ell}\norm{b}_{1}\frac{\Delta t}{\Delta x}<1,
 \\
\rho:&= \max_{j=1,\ldots,d} \left(\mathcal{R}_s(-h-h\gamma\Delta t \lambda_{\text{heat},j},h\Delta t \mu_{\text{adv},j})^{1/2}\right)<1.
\end{align*}}
Then, for all $\varepsilon>0$ there exists $C_\varepsilon$ (independent of the dimension $d$)  such that the iterates $\boldsymbol{x}_k$ of Algorithm \ref{alg:explicitimplicit} converge to the
SDIRK solution $\boldsymbol{Y}_{n+1}$ of \eqref{eq:nonlinearimp} with the following convergence estimate,
    \begin{equation} \label{eq:boundadv}
\|\boldsymbol{x}_k-\boldsymbol{Y}_{n+1}\|_2 \leq C_\varepsilon
\left(\varepsilon+ \min(\rho,\hat \rho)
\right)^k \|\boldsymbol{x}_0-\boldsymbol{Y}_{n+1}\|_2.
\end{equation} 
\end{theorem}
\begin{proof}
    Following the lines of the proof of Theorem \ref{thm:diff}, we obtain
$$
\boldsymbol{x}_{k+1} - \boldsymbol{Y}_{n+1}= (M+\tilde M) (\boldsymbol{x}_{k}- \boldsymbol{Y}_{n+1}).
$$
Here $M=R_{s}(h\mathcal{V}_1) + B_{s}(h\mathcal{V}_1)(h\mathcal{V}_2)$ is exactly the same matrix that appears in Theorem \ref{thm:diff} while the new term $\tilde M=B_{s}(h\mathcal{V}_1)(\mathcal{A}\otimes h\Delta t E)$ arises due to the presence of advection.
Now using the same change of coordinates as in Theorem \ref{thm:diff} we have that 
\begin{align*}
(Q_\varepsilon^{-1}P^{-1}\otimes I_d) (M+\tilde M) (PQ_\varepsilon\otimes I_d) &= I_m \otimes R(h\hat{\mathcal{V}}_1) 
+ \varepsilon F \otimes B(h\hat{\mathcal{V}}_1)h\hat{\mathcal{V}}_2 \\
&+ (\gamma I_m+\varepsilon F)\otimes B(h\hat{\mathcal{V}}_1)h\Delta t E,
\end{align*}
which yields
    \begin{align*} 
\|\boldsymbol{x}_k-\boldsymbol{Y}_{n+1}\|_2 &\leq C_\varepsilon
\Big(
 \|\mathcal{R}_{s}(-hI_d - h\Delta t \gamma D,hE) \|_2 \\
&+\varepsilon\|F\|_2 (\|B_s(h\hat{\mathcal{V}}_1)h\hat{\mathcal{V}}_2\|_2 + h\Delta t\|B_s(h\hat{\mathcal{V}}_1)\|_2\|E\|_2)
\Big)^k \|\boldsymbol{x}_0-\boldsymbol{Y}_{n+1}\|_2.
\end{align*} 
Now proceeding in a similar way as the end of the proof of Theorem \ref{thm:diff},
 we obtain the desired estimate \eqref{eq:boundadv}.
\end{proof}
Note that in the absence of advection ($b=0$), the estimate of Theorem \ref{thm:ad} reduces to the one of Theorem \ref{thm:diff} with $\hat \rho \leq \rho=\alpha_s(\eta)$. 


\begin{figure}[t]
    \centering
    \begin{subfigure}{0.45\linewidth}
        \includegraphics[width=\linewidth]{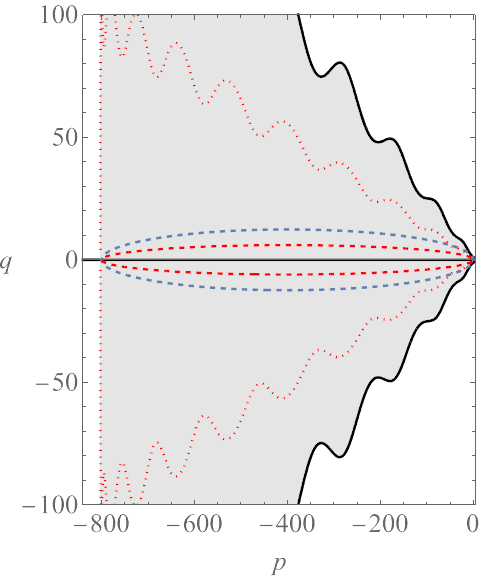}
        \caption{Stability domain}
    \end{subfigure}
    \begin{subfigure}{0.46\linewidth}
            \includegraphics[width=\linewidth]{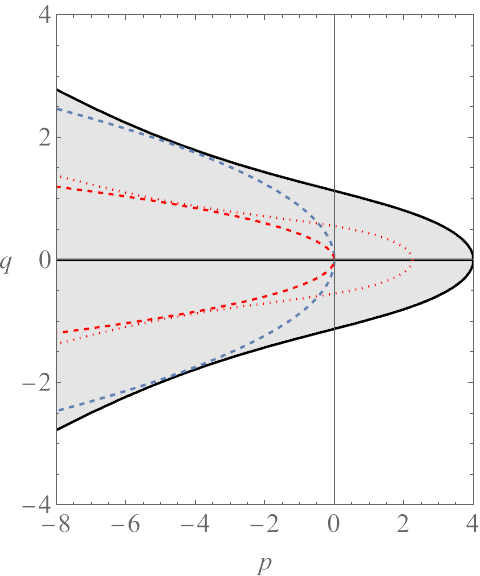}
        \caption{Zoomed version}
    \end{subfigure}
    \caption{{Stability regions $\mathcal{S}_\gamma$ in \eqref{eq:stab_domain} with $\gamma=1/4$ of the diffusion-advection coupling for $s=20$ and damping parameters $\eta=4$. We also include the largest ellipse of the form \eqref{eq:ellipse_pq} tangent to the $y$-axis that can be included in the stability region $\mathcal{S}_\gamma$ in \eqref{eq:stab_domain} (blue dashed lines) where the iteration error decays, and in the stability region $\widehat{\mathcal{S}}_\gamma$ in \eqref{eq:stabdomainmod} where the iteration error is halved at each iteration.}}
    \label{fig:stability1}
\end{figure}
We now plot in Figure \ref{fig:stability1} the stability domain $\mathcal{S}_\gamma$ given by \eqref{eq:stab_domain} with $\gamma=1/4$ for $s=20$ and $\eta=4$. This domain covers a large  portion of the negative real axis as expected from the properties of the Chebyshev polynomials. On the other hand, the width on the $q$-axis near the origin is much narrower.  
{The blue ellipse in Figure \ref{fig:stability1} indicates the largest ellipse of the form \eqref{eq:ellipse_pq} that contains the origin with principal axes parallel to the $p$-$q$ axes and that can be included in the stability domain $\mathcal{S}_\gamma$.  If we denote by $d_s$ and $a_s$  respectively the width and the half-height of this ellipse  we observe that $d_{s}$ grows quadratically with the number of stages $s$ while  $a_{s}$ grows linearly with $s$ as illustrated in Figure \ref{fig:dsas}. More precisely, we find numerically for this range of parameters that $d_s$ and $a_s$ grow like 
\begin{equation} \label{eq:ellipsegrowth}
d_s = 2s^2,\qquad a_s = 0.62s.
\end{equation}
For comparison, we also plot in  in Figure \ref{fig:stability1} in red dotted lines the largest ellipse of the form \eqref{eq:ellipse_pq} that is included in the modified stability domain 
\begin{equation} \label{eq:stabdomainmod}
\widehat{\mathcal{S}}_\gamma = \{(p,q) \in \mathbb{R}^2\ ;\ R_{\text{stab}}(p,q) <1/2 \}
\end{equation}
where the iteration error is at least halved at each iteration using $\hat \rho<1/2$ in \eqref{eq:boundadv}. This is different to blue ellipse corresponding to the stability domain $\mathcal{S}$ in \eqref{eq:stab_domain} where for $\hat \rho<1$ only the decay of the iteration error is guaranteed. This illustrates the favorable stabilization of the explicit stabilised iteration taking advantage of the diffusion term even for advection terms with relatively large Peclet numbers.
}  
{Additionally, note that in the case where ${\rho}<1$ in Theorem \ref{thm:ad} is an appropriate bound for an SDIRK method, the behaviour of the stability function of the implicit Euler studied in Figures \ref{fig:stability1}-\ref{fig:dsas} dictates the choice of parameters, with  the only difference is that} {the stability domain size is scaled with respect to the diffusion {(horizontal $p$-axis in Figure \ref{fig:stability1})} by a factor $\gamma^{-1}$} 
\begin{remark}
{As emphasized in \cite[Remark 3.4]{AV13} for the stability analysis of the PIROCK method for advection-diffusion problems, the linear growth in \eqref{eq:ellipsegrowth} of the ellipse half-height is a feature of the proposed explicit stabilised implementation of SDIRK methods as it allows for relatively large Peclet numbers.
Indeed, considering the advection-diffusion problem \eqref{eq:advdiffexact} with spatial discretization \eqref{eq:adv_diff}.
Substituting $h\Delta t \alpha = 2 s^2$ into $h\Delta t \beta \leq \max(0.62s,1)$  with $h\simeq 1.40$ yields 
$\Delta t \beta \leq \max(\frac{0.62^2}{2\cdot 1.40} \alpha/\beta, 1)$. 
Using $\beta=\norm{b}_{1}\Delta x^{-1}$, $\alpha=4aN\Delta x^{-2}$ we deduce the the stability sufficient condition%
$$\Delta t \leq \max(0.55 aN/\norm{b}_{1}, \Delta x/\norm{b}_{1}),$$
which is a stability constraint on the stepsize $\Delta t$ independent of
the spatial grid size  small enough ($\Delta x \leq 0.55 aN$). This is a feature of the proposed partitioned explicit stabilised implementation, shared with the PIROCK method satisfying an analogous stability sufficient condition (see \cite[Remark 3.4]{AV13}), in contrast to the other stabilised explicit integrators RKC, PRKC and ROCK2 where only small Peclet numbers can be considered.
}
\end{remark}

\begin{figure}[t]
    \centering
    \begin{subfigure}{0.43\linewidth}
        \includegraphics[width=\linewidth, height=4.3cm]{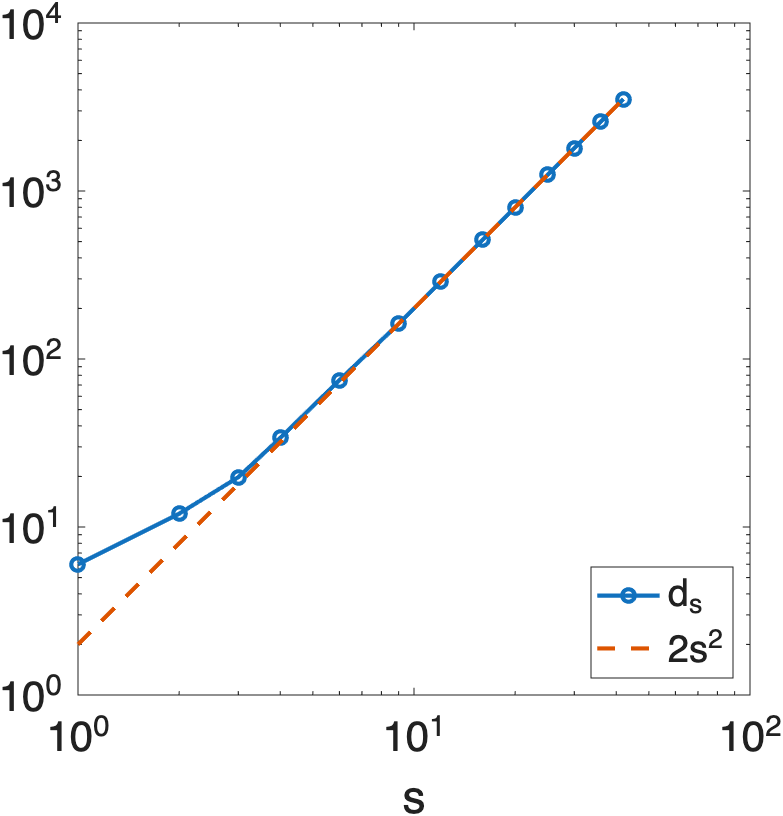}
        \caption{Major axis $d_s$}
    \end{subfigure}
    \begin{subfigure}{0.43\linewidth}
        \includegraphics[width=\linewidth, height=4.3cm]{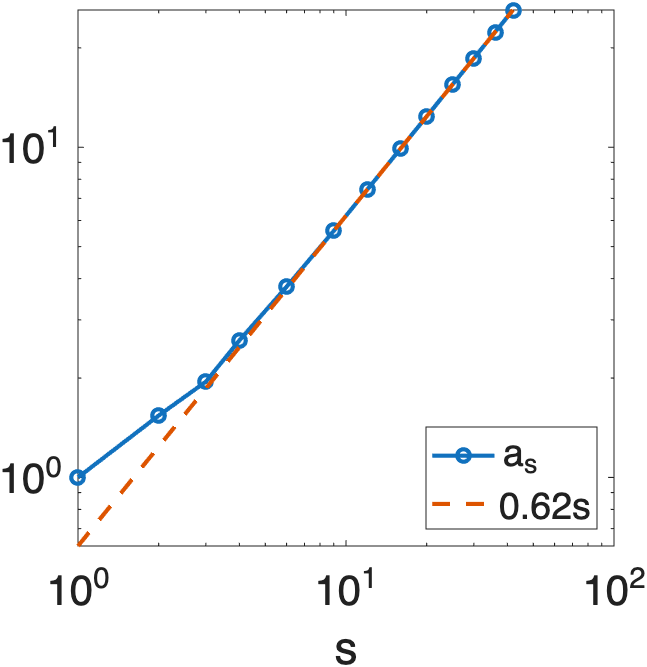}
        \caption{Minor radius $a_s$}
    \end{subfigure}
   
    \caption{{
    Stability of the diffusion-advection coupling for $\gamma=1/4,s=20$ and damping parameters $\eta=4$.
    Growth with respect to the stage parameter $s$ of the width $d_s$ and the half-height $a_s$ in \eqref{eq:ellipsegrowth} of the largest ellipse \eqref{eq:ellipse_pq} tangent to the y-axis that can be included in the stability region $\mathcal{S}_\gamma$ in \eqref{eq:stab_domain} with $\gamma=1/4$.}
    }
    \label{fig:dsas}
\end{figure}

	\section{Numerical tests}
 \label{sec:num}
In this section, we present a number of  numerical experiments on the Brusselator problem for the SDIRK4 method. More precisely, we study the case of a highly stiff reaction as well as a large advection setting.  We denote with exSDIRK4  the proposed explicit stabilised implementation of SDIRK4, with \textsc{Matlab} prototype codes made publicly available at \cite{codesMatlabexSDIRK26}, and we compare its performance with the order 2 PIROCK method. Note that established explicit methods of order $4$ such as ROCK4 \cite{Abd02}, are not suitable for the two problems studied here due to the stiff nature of the advection and reaction terms. In all our numerical experiments, the time step $\Delta t$ is chosen adaptively as described in the remark below. 
\begin{remark} 
Compared to the standard embedded error estimator for stiff problems recalled in Remark \ref{rem:variable_step}, see also \cite[Chap. IV.6]{HaW96}, the only difference for the explicit stabilised implementation is that the inverse Jacobian term $J^{-1}$ involved in \eqref{eq:shampine} is no longer available. This is indeed the main feature of the proposed approach to avoid such a possibly large dimension Jacobian and its inverse, which can raise linear algebra issues. Alternatively to \eqref{eq:shampine}, we implement the Shampine trick using the following error estimator of the exSDIRK4 method,
$$
err={err}_s
$$
where we define $\overline {err}= J^{-1}(y_1-\hat y_1)$ where $J = I_{d} - h\, \Delta t \gamma F'_R(y_1)$ is the Jacobian involved for the reaction term, and the ${err}_j$ are computed by induction on $j\in \{1,\cdots,s\}$ using ${err}_0 = {err}_{-1} = 0$ as
$${err}_j = \mu_j \left(h \Delta t \gamma (F_D(y_1+{err}_{j-1}) - F_D(y_1))+ \overline{err}\right)+
\nu_j {err}_{j-1}-(\nu_j-1) {err}_{j-2},$$ 
where $\mu_1 = \omega_1/\omega_0, \nu_1=1$ and $\mu_j,\nu_j$ and the stage number $s$ are defined in Algorithm \ref{alg:3} (see \eqref{eq:rkcd_iter}).
We emphasize that the above error estimator benefits from $L$-stability with respect to diffusion and reaction terms in the spirit of the the standard estimator \eqref{eq:shampine}. Indeed, firstly in the absence of diffusion and reaction ($F_D=F_A=0$) it reduces to $err=\overline{err}=J^{-1}(y_1-\hat y_1)$ which exactly coincides with \eqref{eq:shampine}, and secondly in the case of a linear diffusion $F_D=Dy$, it yields $err_s=B_s(h\Delta t \gamma D) J^{-1} (y_1-\hat y_1)$ analogously to \eqref{eq:it_rkcd1} and where we recall the bound $|B_s(z)|\leq C/|z|$.
\end{remark}


Throughout the numerical experiments we make the following implementation assumption: the evaluations of $F_D[\boldsymbol Y]$, $F_A[\boldsymbol Y]$, and $F_R[\boldsymbol Y]$ for the $m$ internal
stages, are computed in \emph{parallel}. This is natural given the
block-diagonal structure of the stabilisation operators $R_s(h\mathcal{V}_1)$ and
$B_s(h\mathcal{V}_1)$, which decouple across stages, and reflects the remark following
Theorem~\ref{thm:diff}. For a vector $v\in R^d$ that represents a discretized function $f:\R^d\rightarrow\R^d$, we define the discrete $L^2-$norm by
\[\|v\|_{L^2} = \sqrt{\frac{\sum_{i=1}^{d}v_i^2}d}.\]
\begin{remark}
    As we mentioned before, for simplicity, we use $\mathbbm{1} \otimes y_n$ as initial guess in Algorithm \ref{alg:3} and, to coincide with the true solution of SDIRK4, we choose a fixed $tol_{it} = 10^{-12}$ as optimisation tolerance. However, the cost of exSDIRK4 discussed in the numerical tests below can be further reduced by choosing a better initial guess using interpolation or dense output techniques, and by choosing $tol_{it}$ smaller than the relative tolerance by only one or two orders of magnitude.
\end{remark}
	\subsection{1D Brusselator problem with highly stiff reaction}
    
	\begin{table}[b!]
		\centering
		\begin{tabular}{|c|c|c|c|c|c|c|}
			\hline
			Method & $tol$   &  $F_D$ evals & $F_R$ evals &  $F'_R$ evals &  acc(rej) & error \\
			\hline
			exSDIRK4 & $10^{-1}$ & $4447$ & $245$ & $40$ & $20(0)$ & $3.5\times 10^{-4}$\\
			PIROCK & $10^{-1}$ & $548$ & $63$ & $12$ & $12(0)$ & $2\times 10^{-2}$ \\
			\hline
			exSDIRK4 & $10^{-3}$ & $5072$ & $559$ & $77$ & $42(3)$ &$2.0\times 10^{-5}$ \\
			PIROCK & $10^{-3}$ & $870$ & $123$ & $24$ & $24(0)$ & $3.3\times 10^{-4}$ \\
			\hline
			exSDIRK4 & $10^{-5}$ & $7718$ & $1307$ & $183$ & $102(3)$ & $1.4\times10^{-6}$ \\
			PIROCK & $10^{-5}$ & $5719$ & $6205$ & $1239$ & $1234(6)$ & $9.9\times10^{-6}$\\
			\hline
			exSDIRK4 & $10^{-7}$ & $13463$ & $3053$ & $439$ & $231(4)$ & $1.7\times10^{-8}$\\
			PIROCK & $10^{-7}$ & $41608$ & $54967$ & $10992$ & $10987(6)$ & $1.4\times10^{-7}$\\
			\hline
			exSDIRK4 & $10^{-9}$ & $23056$ & $6248$ & $915$ & $468(5)$ & $2.2\times10^{-10}$\\
			PIROCK & $10^{-9}$ & $191700$ & $279447$ & $55888$ & $55888(2928)$ & $1.0\times10^{-9}$\\
			\hline
		\end{tabular}
		\captionof{table}{{Comparison between our explicit stabilised implementation of SDIRK4 and PIROCK for different} {prescribed tolerances.} }
		\label{table:bruss}
	\end{table}
	Consider the following Brusselator problem
	\begin{equation}\label{eq:bruss1d}
		\begin{split}
		\partial_t u &= \alpha\Delta u + A +u^2v - (B+1)u,\\
		\partial_t v &= \alpha\Delta v -u^2v + Bu,\\
        u(0,t)&=u(1,t), \quad v(0,t)=v(1,t),
		\end{split}
	\end{equation}
	where $x\in (0,1)$, $t\in (0,T)$, $\alpha=0.2$, $A=1$, $B=3\times 10^7$, and $T=1$. We consider the following initial condition
	$$
	u(x,0)=1+\sin(2\pi x) \quad \text{and} \quad v(x,0)=3, 
	$$
    and periodic boundary conditions.
	We use second order finite differences for the space discretisation with mesh size $\Delta x =1/200$, and the starting time step $\Delta t_0=10^{-6}$. 

    \begin{figure}[htbp]
    \centering
    \begin{subfigure}[b]{0.325\textwidth}
        \centering
        \includegraphics[width=\textwidth]{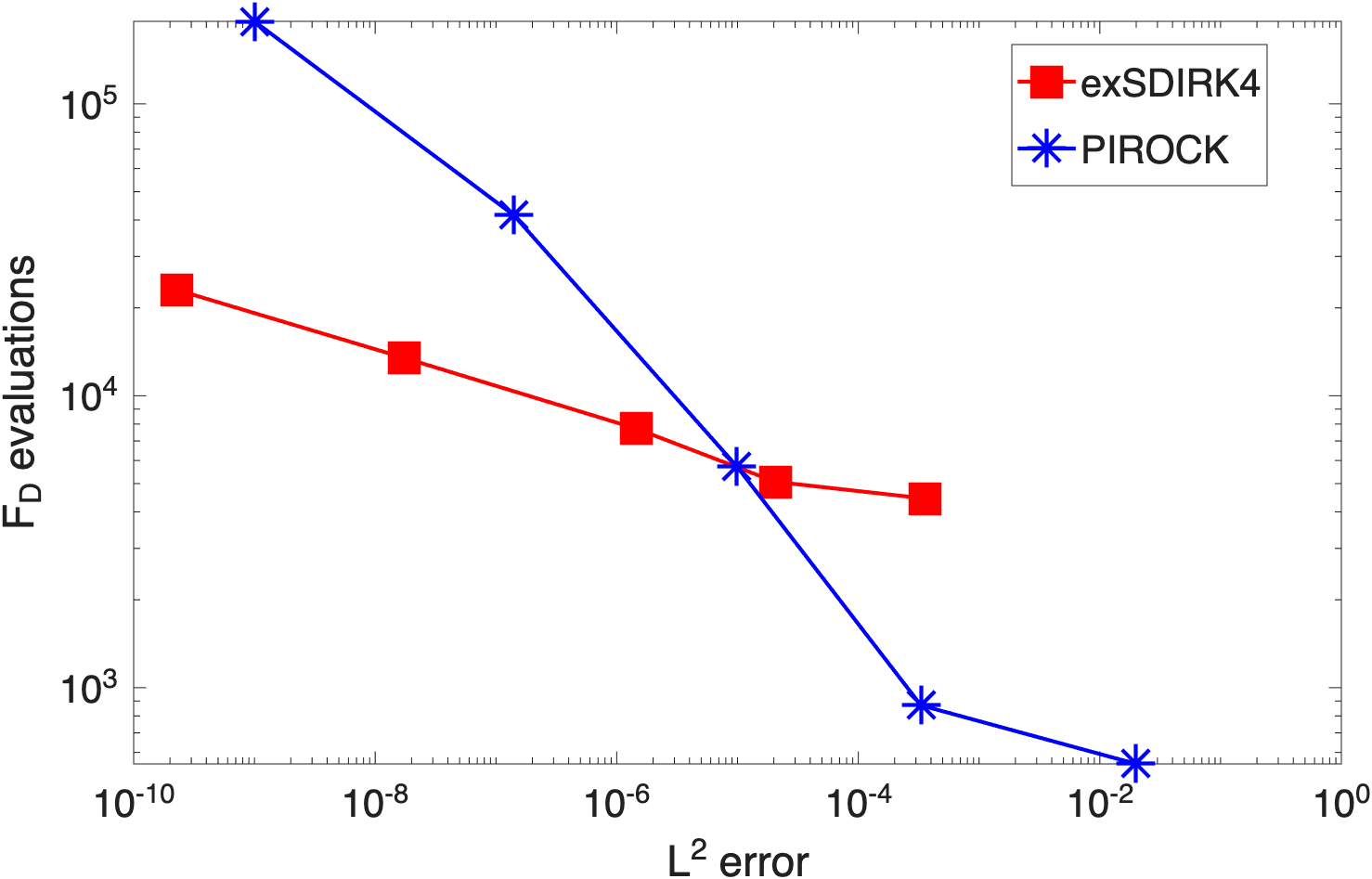}
        \caption{$F_D$ evaluations.}
    \end{subfigure}
    \hfill
    \begin{subfigure}[b]{0.325\textwidth}
        \centering
        \includegraphics[width=\textwidth]{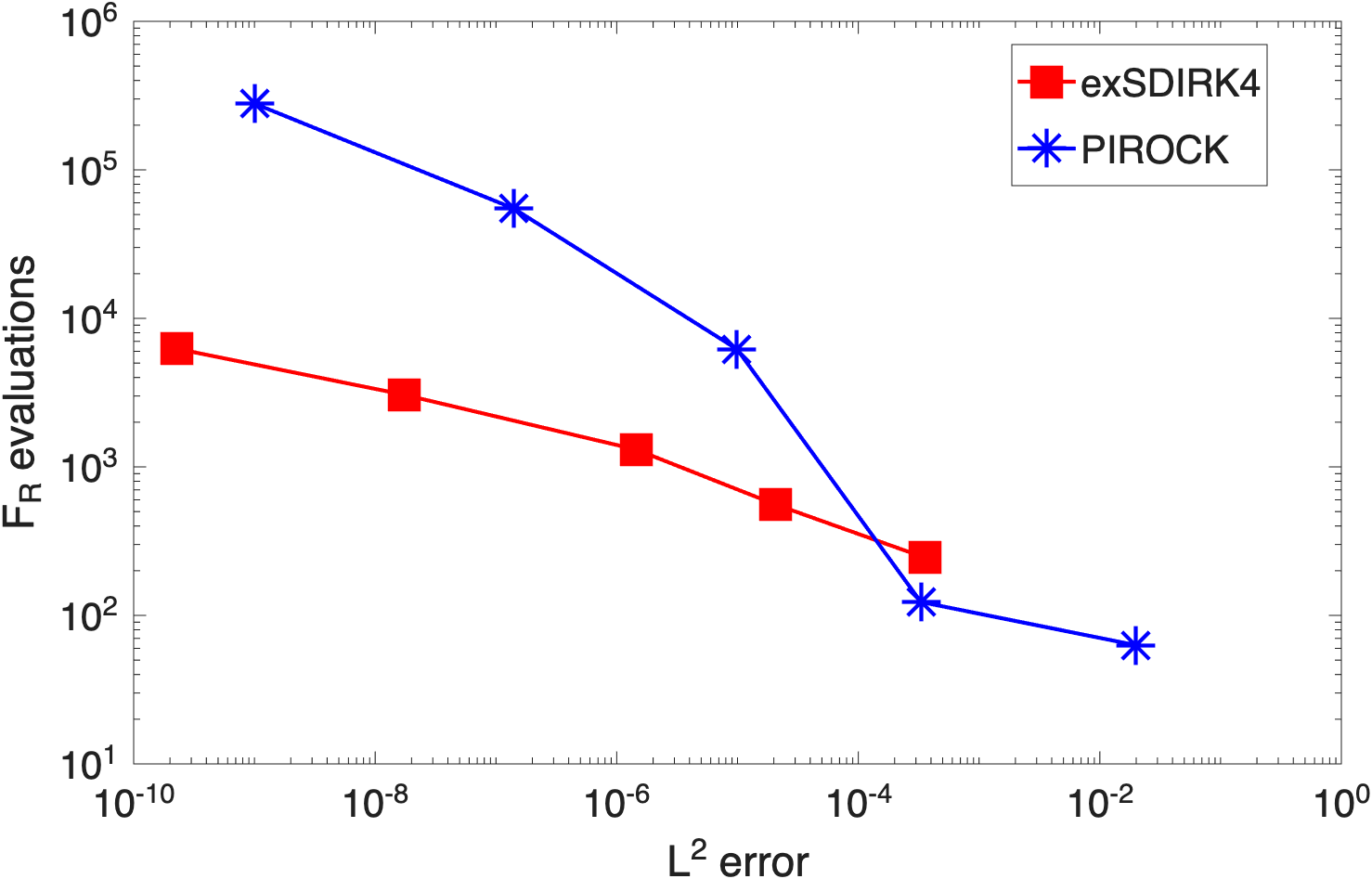}
        \caption{$F_R$ evaluations.}
    \end{subfigure}
    \hfill
    \begin{subfigure}[b]{0.325\textwidth}
        \centering
        \includegraphics[width=\textwidth]{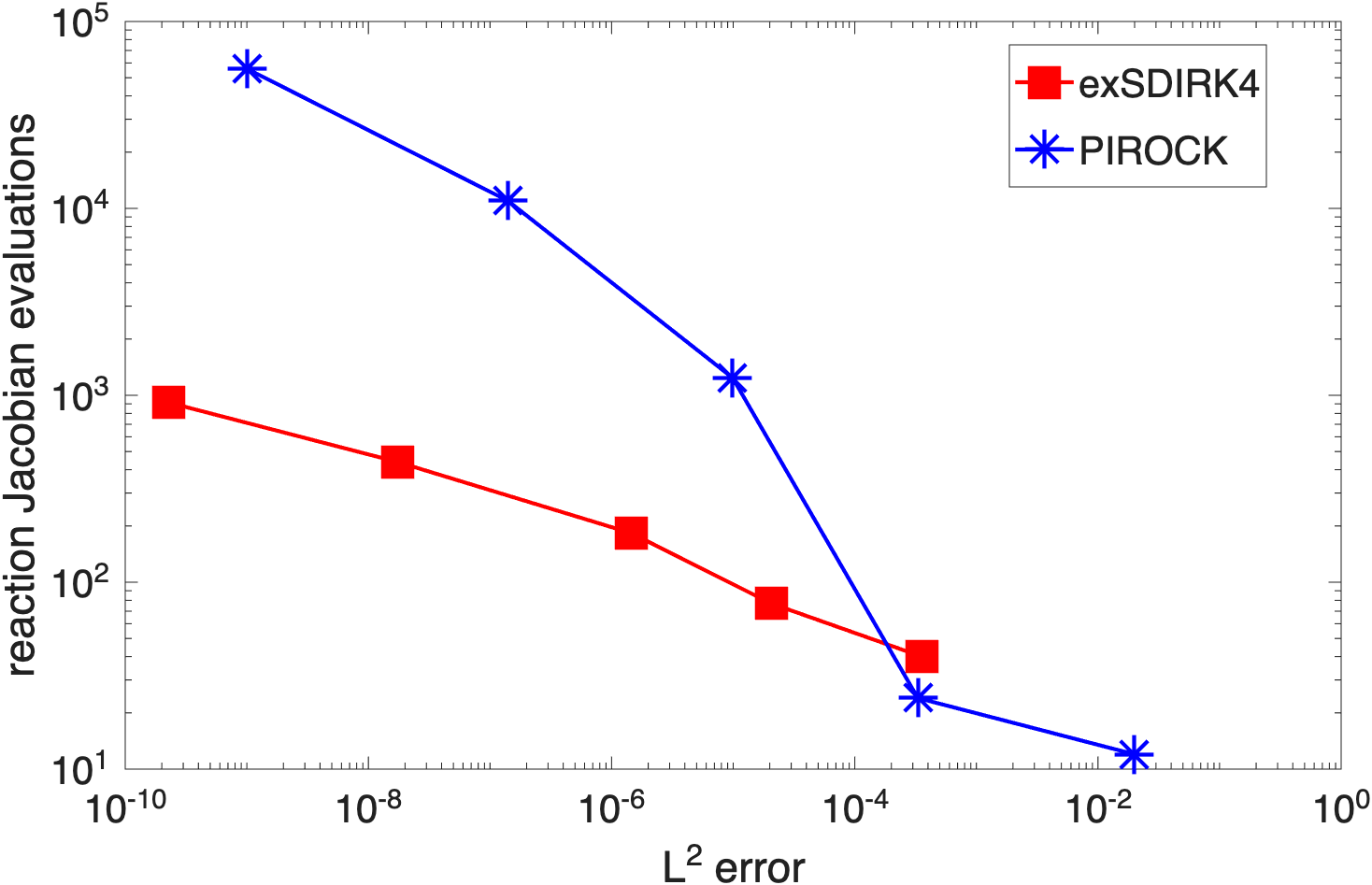}
        \caption{{$F_R'$ evaluations}}
    \end{subfigure}

    \caption{Cost comparison between exSDIRK4 and PIROCK for Problem \eqref{eq:bruss1d}}
    \label{fig:bruss1d}
\end{figure}

Table~\ref{table:bruss} and Figure~\ref{fig:bruss1d} report the cost comparison between exSDIRK4 and PIROCK for the 1D Brusselator problem {\eqref{eq:bruss1d}}. For loose tolerances ($tol = 10^{-1}$ and $10^{-3}$), PIROCK requires fewer function evaluations overall; however, exSDIRK4 already delivers a substantially lower error for the same requested tolerance, reflecting its higher order of accuracy. As the tolerance tightens, the balance shifts decisively in favour of exSDIRK4. At $tol = 10^{-5}$, both methods have a comparable number of $F_D$ evaluations, but exSDIRK4 requires more than $4$ times fewer $F_R$ evaluations and $6$ times fewer $F'_R$ {evaluations/inversions}. This advantage grows rapidly: at $tol = 10^{-7}$, PIROCK needs approximately $3$ times more $F_D$, $18$ times more $F_R$, and $24$ times more $F'_R$ {evaluations/inversions} than exSDIRK4. At the tightest tolerance $tol = 10^{-9}$, PIROCK accumulates 55,888 accepted steps with 2,928 rejections, while exSDIRK4 completes the integration in only 468 accepted steps with 5 rejections—a reduction of more than two orders of magnitude in the number of steps. These trends are clearly visible in Figure~\ref{fig:bruss1d}, which displays the number of $F_D$, $F_R$, and $F_R'$ evaluations as a function of the tolerance. The work-precision curves show that the crossover point beyond which exSDIRK4 is uniformly cheaper occurs around $tol \approx 10^{-4}$, and the gap widens as the tolerance decreases, which is consistent with the high-order convergence of exSDIRK4.

\subsection{2D Brusselator with large advection}
{We now consider the following Brusselator model in 2D:}
\begin{equation}\label{eq:bruss2d}
		\begin{split}
		\partial_t u &= \nu\Delta u + \mu U\cdot \nabla u + A +u^2v - (B+1)u,\\
		\partial_t v &= \nu\Delta v + \mu V\cdot \nabla v -u^2v + Bu,\\
		\end{split}
	\end{equation}
	where $x\in (0,1)^2$, $t\in (0,T)$, $\nu=0.1$, $\mu=2$ $A=1.3$, $B=10^7$, $U=(-0.5,1)^T$, $V=(0.4,0.7)^T$ and $T=0.5$. Following \cite{AV13}, we choose
	$
	u(x,0)=22x_2(1-x_2)^{3/2} , \quad v(x,0)=27x_1(1-x_1)^{3/2}, 
	$
	and periodic boundary conditions. For the space discretisation, we use again second order finite differences with mesh size equals to $10^{-2}$ in both dimensions and the starting time step $\Delta t_0=10^{-6}$.

\begin{figure}[htbp]
    \centering
    \begin{subfigure}[b]{0.325\textwidth}
        \centering
        \includegraphics[width=\textwidth]{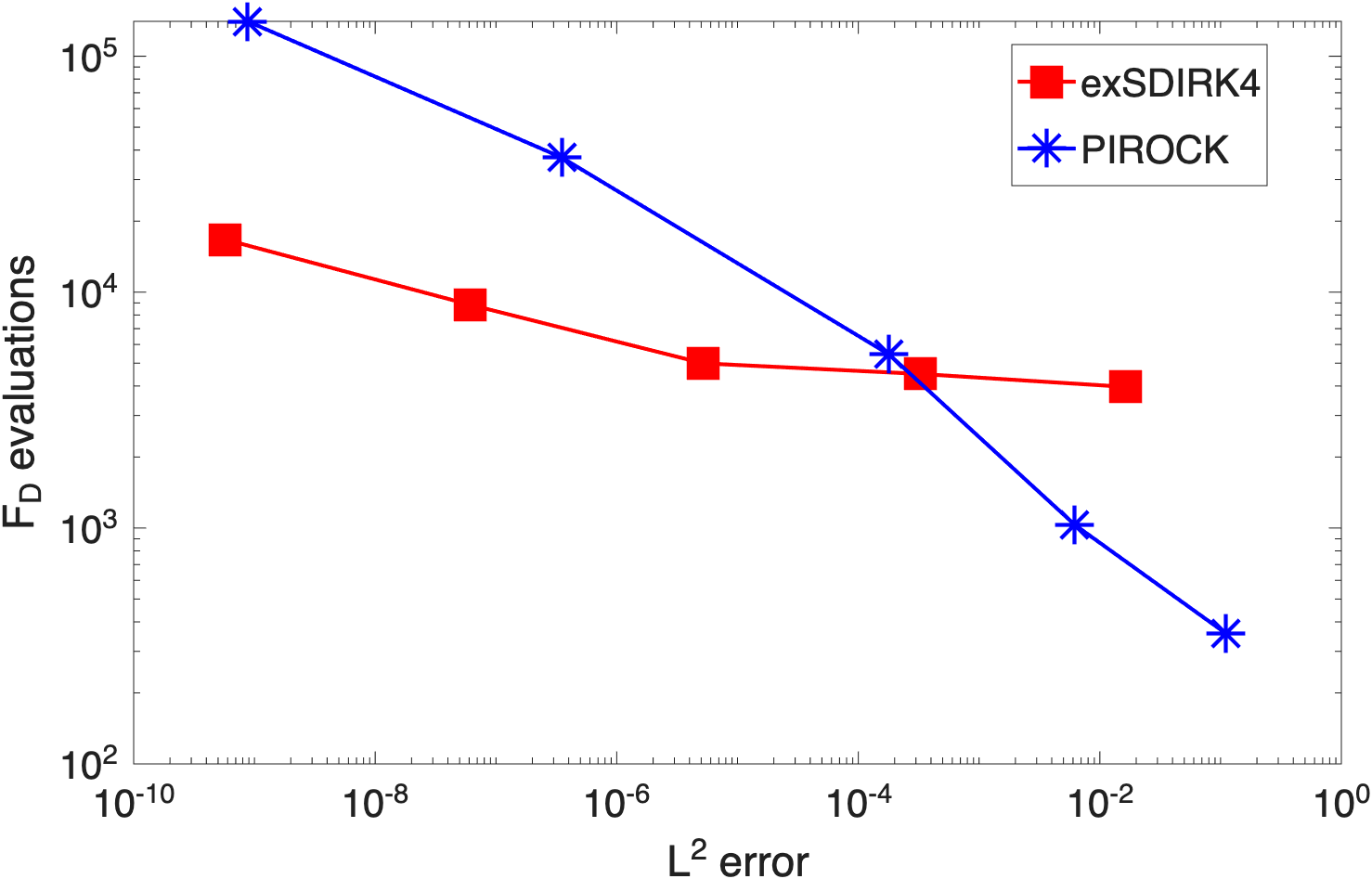}
        \caption{$F_D$ evaluations.}
    \end{subfigure}
    \hfill
    \begin{subfigure}[b]{0.325\textwidth}
        \centering
        \includegraphics[width=\textwidth]{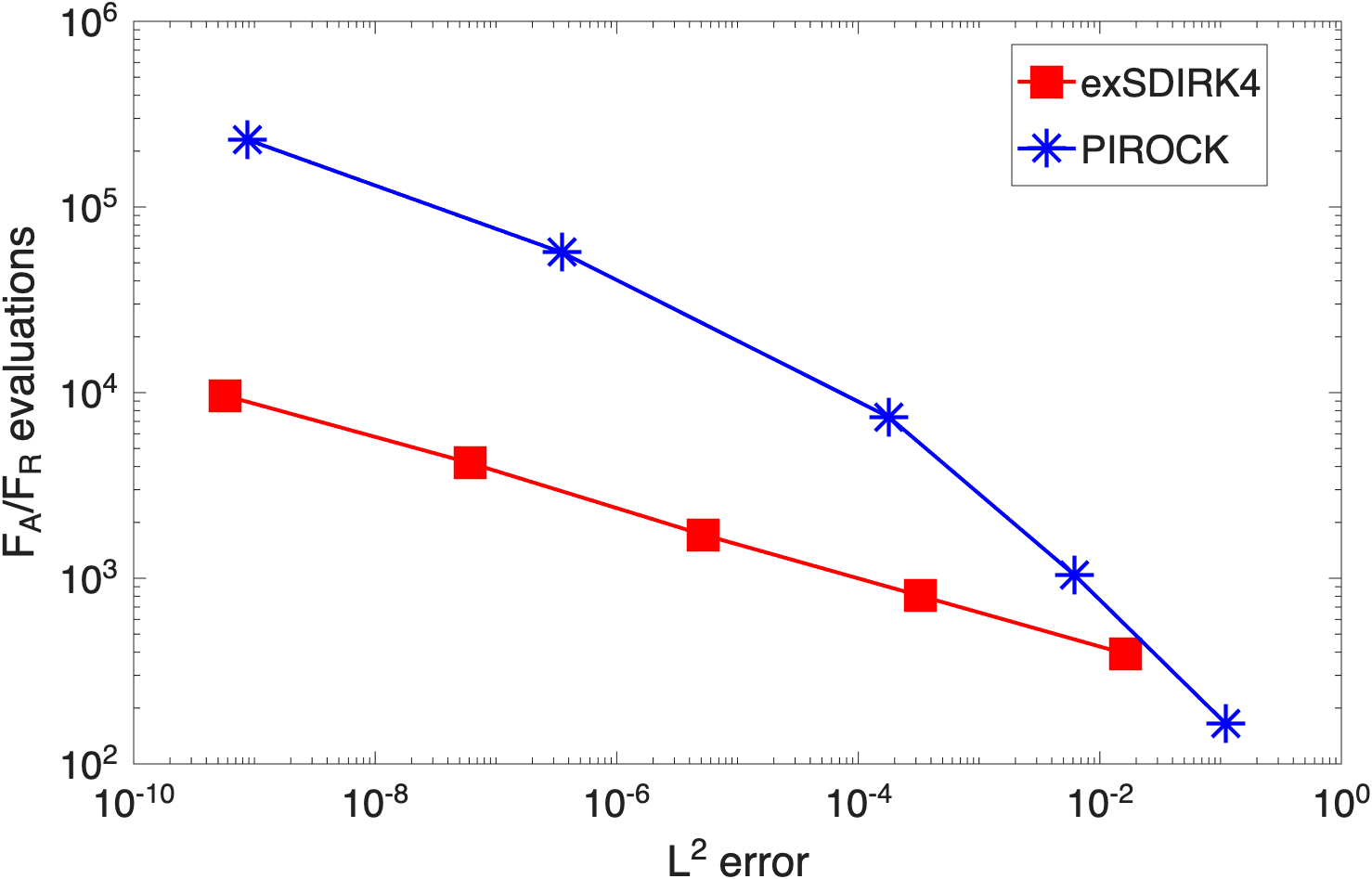}
        \caption{$F_A$ and $F_R$ evaluations.}
    \end{subfigure}
    \hfill
    \begin{subfigure}[b]{0.325\textwidth}
        \centering
        \includegraphics[width=\textwidth]{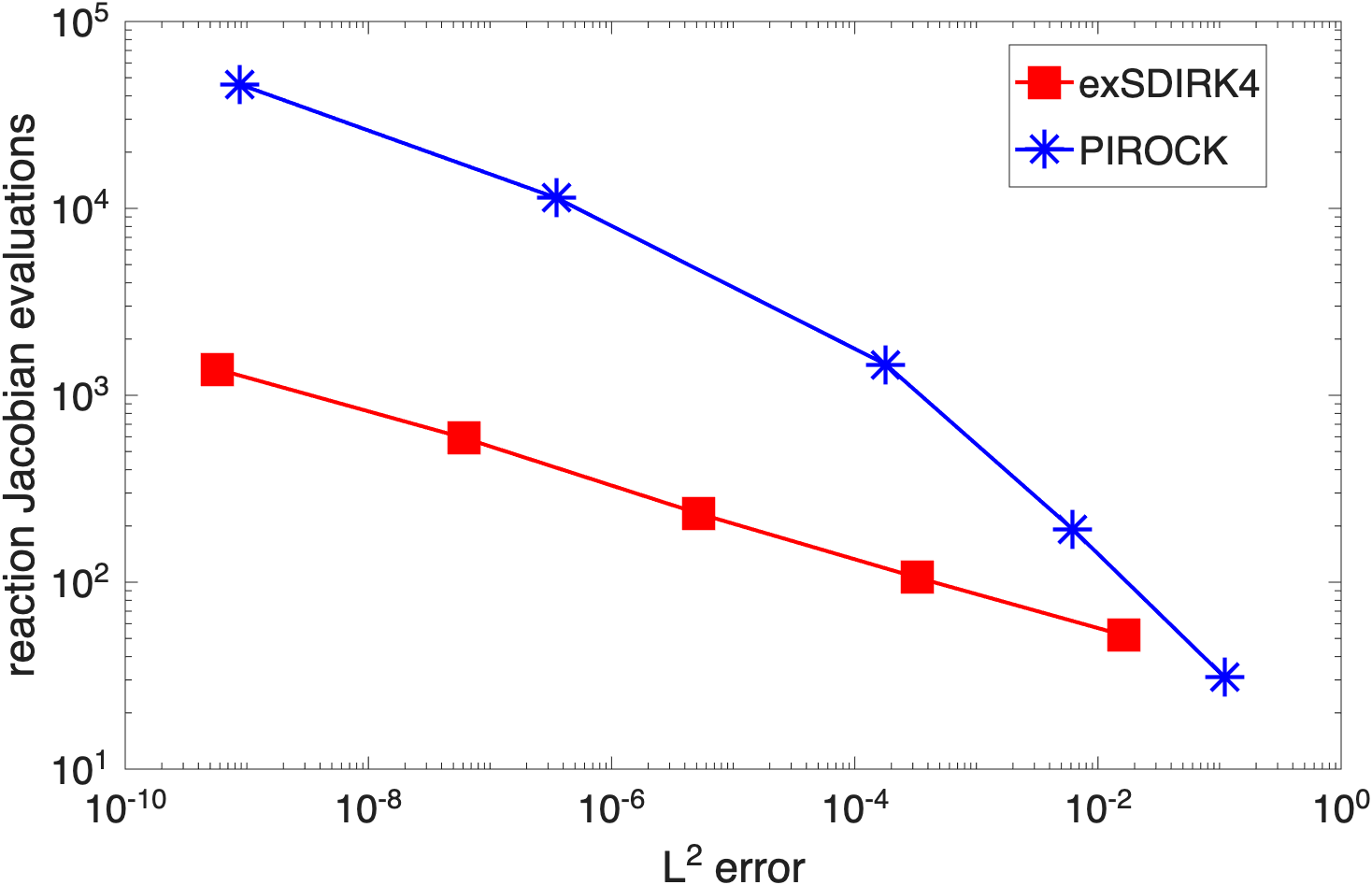}
        \caption{{$F_R'$ evaluations.}}
    \end{subfigure}
    \caption{Cost comparison between exSDIRK4 and PIROCK for Problem \eqref{eq:bruss2d}}
    \label{fig:bruss2d}
\end{figure}

Figure~\ref{fig:bruss2d} reports analogous results for the 2D Brusselator problem with large advection. The qualitative behaviour is very similar to the 1D case: exSDIRK4 is outperformed by PIROCK for loose tolerances in terms of raw function evaluation counts, but becomes significantly more efficient as the tolerance tightens. The three panels—$F_D$ evaluations, $F_A$/$F_R$ evaluations, and $F_R'$ evaluations—all exhibit the same crossover pattern observed in Figure~\ref{fig:bruss1d}, confirming that the efficiency advantage of exSDIRK4 is not an artefact of the one-dimensional setting but carries over to higher-dimensional problems with combined diffusion, advection, and stiff reaction terms.

\section{Conclusion}


We introduced an explicit stabilised implementation of singly diagonally implicit
Runge-Kutta (SDIRK) methods for advection-diffusion-reaction PDEs, with the fourth-order
method SDIRK4 as the primary example. The central idea is to reformulate the implicit
Runge-Kutta stage equations \eqref{eq:nonlinearimp} as the steady state of an auxiliary
ODE \eqref{eq:gf1}, and to compute that steady state by a partitioned
Runge-Kutta-Chebyshev iteration (Algorithm~\ref{alg:3}) inspired by the RKCD
optimisation method of \cite{EVVZ21}. The partition of the Runge-Kutta coefficient matrix
$\calA = \gamma I_m + (\calA - \gamma I_m)$ is essential: it isolates a symmetric,
block-diagonal stiff part $G_D$ that governs the stability of the iteration through the
stability polynomial $R_s$, avoiding the non-diagonalisable amplification matrix that
would otherwise prevent convergence, as analysed in detail in Section~\ref{appendix:naive}.
The off-diagonal diffusion and advection contributions enter through the damping polynomial
$B_s$, whose extended stability along the imaginary axis accommodates large advection
velocities under a mild CFL-like condition on the iteration step size $h$ — significantly
more permissive than the classical CFL restriction on the physical time step $\Delta t$.
The stiff reaction term is handled implicitly at each inner iteration via a local inversion
of $J = I_{md} - hG'_R(\boldsymbol{x}_k)$, which remains computationally inexpensive
because the reaction Jacobian carries no spatial coupling between grid points and is
therefore block-diagonal. The result is a method that requires no large-scale Newton solves
or global linear algebra, and yet achieves the high-order accuracy and strong stability of
a fourth-order $L$-stable integrator. The only implicit component is this local reaction
Jacobian inversion, in sharp contrast to the dense, high-dimensional Jacobian systems that
a standard Newton-based implementation of SDIRK4 would require.

The convergence and stability properties of the new method are rigorously established in
two settings. In the pure linear diffusion case (Theorem~\ref{thm:diff}), we prove that
Algorithm~\ref{alg:3} converges with {decay factor} $\alpha_s(\eta)+\varepsilon$ for any
$\varepsilon>0$, independently of the spatial dimension $d$, and that implementing one
step of an $m$-stage SDIRK method requires only
$\mathcal{O}(\sqrt{\Delta t}/\Delta x)$ evaluations of $F_D$ — the same asymptotic cost
as a standard first-order explicit stabilised method. In the linear advection-diffusion
case (Theorem~\ref{thm:ad} and Propositions~\ref{prop:rate1}--\ref{prop:rate2}), we show
that convergence is maintained provided the advection-to-diffusion ratio satisfies a
mild CFL-like condition, a condition that is significantly
more permissive than the one imposed by a classical explicit scheme on the physical time
step $\Delta t$.

Assuming parallel implementation of the function evaluations across the method stages,
the numerical experiments on the 1D and 2D Brusselator problems with {highly stiff
reaction  confirms} the theoretical predictions and demonstrate a clear
advantage over the second-order method PIROCK \cite{AV13} whenever tight tolerances are
required. The method occupies a natural middle ground: more implicit than a pure explicit
stabilised method such as PIROCK, but far cheaper than a fully Newton-based implicit
solver, and this balance is precisely what drives its efficiency at high accuracy demands.

The present framework relies on the lower-triangular, singly diagonal structure of SDIRK
methods to decouple the implicit system stage by stage and to identify the symmetric
stabilising part $G_D$. A natural and challenging next step would be to extend the approach to
fully implicit Runge-Kutta methods such as Radau IIA or Gauss-Legendre collocation
schemes. 
{A major difficulty for their implementation is their full matrix structure of the Runge-Kutta coefficients $\calA$, and studied for instance in \cite{IM98,BIM15} where splitting approaches based on  triangular matrices and low rank linear algebra techniques are proposed. 
In such fully implicit case,} the
Runge-Kutta matrix $\calA$ is a full, non-triangular matrix with non-real eigenvalues,
so the partitioning strategy proposed in this paper and the associated convergence analysis do not apply straightforwardly.

\paragraph{Acknowledgments} 
The authors would like to thank Ernst Hairer and Francesca Mazzia for helpful discussions on the implementation of SDIRK and RADAU methods. G.V. was partially supported by the Swiss National Science
Foundation, projects No. 200020\_214819, No. 200020\_192129 and No. 10009199. I.A. was partially supported by the Swiss National Science Foundation, projects No. P500PT\_210968 and No. P5R5PT\_222213.

   \bibliographystyle{abbrv}
	\bibliography{refs}

\appendix

\section{Naive implementation}
\label{appendix:naive}
We now explain in more detail the issues observed in Figure \ref{fig:iterations} with respect to the naive stabilised explicit implementation of SDIRK4 applied without partitioning {(corresponding to Algorithm \ref{alg:3} with 
the definition of $G(\boldsymbol{x})$ replaced by $G(\boldsymbol{x})=G_D(\boldsymbol{x})+G_A(\boldsymbol{x})$ with $G_R(\boldsymbol{x})=0$, or equivalently applying Algorithm \ref{alg:rkcd} to equation \eqref{eq:gff} with $F_A=F_R=0$)}. 
Considering for simplicity only one mode of the discrete linear heat equation, corresponding to the Dahlquist test equation $\dot y =\lambda y$ 
with $\lambda=-\Delta t \gamma \lambda_{\text{heat},j}<0$ corresponds to a possibly large parameter $|\lambda|\gg 1$.
We then write one-step recursion induced by Algorithm \ref{alg:rkc} applied to a SDIRK method with $m>1$ internal stages, namely  
\[
\mathbf{Y}_{n+1}=R_{s}(-hI_m+ \lambda h\mathcal{A})\mathbf{Y}_{n}
\]
where $R_{s}$  given by \eqref{eq:stabcheb} is the stability function of RKC, and $h$ is the time step used by the RKC iteration. 
Consider the Jordan form of the matrix $\mathcal{A}=P\mathcal{J}P^{-1}$ where
$\mathcal{J}$ is a Jordan block defined in \eqref{eq:defJ}.
Using the following formula for any analytic function $f$ applied to the Jordan matrix $J$,
        $$
f(\calJ)=\begin{pmatrix}f(\gamma)&f^\prime(\gamma)&\cdots&\frac{f^{(m-1)}(\gamma)}{(m-1)!}\\&f(\gamma)&\ddots&\vdots\\&&\ddots&f^\prime(\gamma)\\0&&&f(\gamma)\end{pmatrix},
        $$
        we deduce taking $f(\calJ)=R(-hI_m+\gamma z \calJ)$ with $z=h\lambda$ and $\hat z=-h+\gamma z$,
$$
P^{-1}R_s(-hI_m +\lambda h \mathcal{A})P = \begin{pmatrix}
R_s(\hat z)&R_s^\prime(\hat z) \gamma z&\cdots& \frac{R_s^{(m-1)}(\hat z)(\gamma z)^{m-1}}{(m-1)!} \\&R_s(\hat z)&\ddots&\vdots\\&&\ddots&R_s^\prime(\hat z) \gamma z\\0&&&R_s(\hat z)
\end{pmatrix}
$$
and it turns out that the off-diagonal terms $\frac{R_s^{(j)}(\hat z)(\gamma z)^{j}}{j!}, 1\leq j\leq m-1$ are arbitrarily large for large $|z|$ for stiff problems where $|\lambda| \gg 1$. 
This is due to the $z^j$ factors, growing like $\bigo(s^{2j})$, and can be shown that the 
derivatives $R_s^{(j)}$ of the stability function oscillate with an amplitude of size $\bigo(s^{-j})$ which is not sufficient to damp the $z^j$ terms.
We obtain  
$$\sup_{s\geq 1, \hat z\in [-L_{s,\eta},0]} \|R_s(-hI_m+z \mathcal A)\|=+\infty$$
and this makes the convergence factor of the naive implementation without partitioning non-uniformly bounded 
as the stiffness increases when refining the spatial mesh ($\Delta x\rightarrow 0$), as illustrate in Figure \ref{fig:iterations}.
Note that in the proof of the main Theorem \ref{thm:diff}, we used the same Jordan decomposition of the Runge-Kutta matrix
$
\calA = P\mathcal{J}P^{-1}
$, with the main difference in the design of the proposed method with partitioning and its stability analysis that the stability function $R_s$ of the explicit stabilised Chebyshev method is always applied to a diagonalizable matrix.

	\end{document}